%% file: main.tex
\documentclass[11pt, a4paper]{amsart}

\usepackage{amsmath, amsthm, amssymb}
\usepackage{txfonts, mathrsfs}
\usepackage[usenames]{color}
\usepackage{psfrag}
\usepackage{graphicx}
\usepackage{todonotes}
\usepackage{comment}
\usepackage[normalem]{ulem}
\usepackage[hidelinks, pdfencoding=auto, psdextra]{hyperref}
\usepackage{cleveref}
\usepackage{microtype}
\usepackage{bbm}
\usepackage{booktabs}
\usepackage{longtable}

\usepackage[final]{showlabels}

\numberwithin{equation}{section}

\newtheorem{thm}{Theorem}[section]

\newtheorem{lem}[thm]{Lemma}

\newtheorem{prop}[thm]{Proposition}

\theoremstyle{definition}
\newtheorem{defn}[thm]{Definition}
\newtheorem{example}[thm]{Example}

\theoremstyle{remark}
\newtheorem{rmrk}[thm]{Remark}
\newtheorem{quest}[thm]{Question}

\newcommand{\ie}{\emph{i.e.}}

\newcommand{\ind}{\mathbbm{1}}

\newcommand{\N}{\mathbb{N}}

\newcommand{\R}{\mathbb{R}}

\newcommand{\bM}{\mathbf{M}}
\newcommand{\bN}{\mathbf{N}}
\newcommand{\bF}{\mathbf{F}}
\newcommand{\bP}{\mathbf{P}}
\newcommand{\bCH}{\mathbf{CH}}

\newcommand{\scrP}{\mathscr{P}}
\newcommand{\scrS}{\mathscr{S}}

\newcommand{\scrH}{\mathscr{H}}
\newcommand{\scrD}{\mathscr{D}}
\newcommand{\scrM}{\mathscr{M}}
\newcommand{\wedgeop}{\wedge}

\DeclareMathAlphabet\euscr{U}{eus}{m}{n}
\newcommand{\niceBV}{\euscr{BV}}

\DeclareMathOperator{\diam}{diam}
\DeclareMathOperator{\dist}{dist}
\DeclareMathOperator{\diver}{div}
\DeclareMathOperator{\reg}{reg}

\DeclareMathOperator{\rmLip}{Lip}

\begin{document}

\title{Non-absolute integration and application to Young~geometric integration}
\keywords{}
\subjclass[2020]{49Q15, 60L99, 28A75, 26A39}

\author[Ph. Bouafia]{Philippe Bouafia}

\address{F\'ed\'eration de Math\'ematiques FR3487 \\
  CentraleSup\'elec \\
  3 rue Joliot Curie \\
  91190 Gif-sur-Yvette
}

\date{\today}

\email{philippe.bouafia@centralesupelec.fr}

\begin{abstract}
We survey several non-absolutely convergent integrals, including the Henstock–Kurzweil and Pfeffer integrals, and use ideas from these theories to investigate the problem of multidimensional Young integration. We further present results on Young geometric integration, namely the integration of certain generalized differential forms over $m$-dimensional subsets of $\R^d$. This is achieved by introducing appropriate notions of chains and cochains, in the spirit of Whitney's geometric integration theory.
\end{abstract}

\maketitle
\tableofcontents{}

\section{Introduction}

The Henstock–Kurzweil integral is a conditionally convergent integral that allows one to integrate, on $[0, 1]$, functions that are more general than those that are Lebesgue integrable. Although it is equivalent to the Denjoy–Perron integral, its discovery in the late 1950s generated considerable interest, as it lends itself more naturally than its predecessors to generalizations in dimensions $d \geq 2$.

In this paper, we survey several extensions of the Henstock–Kurzweil theory, with particular emphasis on the Pfeffer integral, which enjoys a number of remarkable properties. Several accounts of these theories already exist in the literature \cite{Pfef, Pfef2, DePauw2004_divergence}; here, our aim is instead to focus on their functional-analytic aspects. A typical feature of non-absolutely convergent integrals is that they give rise to spaces of integrable functions that are normed but not complete, which makes their analysis delicate. In the case of the Pfeffer integral, the completion of this space can be identified with a space of distributions of order $-1$ known (in the context of Pfeffer integration) as charges, which will play a central role throughout this article.

Starting from Section~\ref{sec:sch}, we show how to construct higher-dimensional Young-type integrals by exploiting the properties of the space $CH([0,1]^d)$ of strong charges.
Several extensions of the one-dimensional Young integral have already been proposed in the literature (see, for instance, \cite{Towghi2002, QuerSardanyons_Tindel_2007_wave_fractional_Brownian_sheet, ChoukGub2014, Harang2020}).
The approach we take as a reference is the Züst integral, which allows for the integration of Hölder differential forms:

\begin{thm}[Züst, 2011~\cite{Zust}]
Let $f \in C^\gamma([0,1]^d)$ and $g_i \in C^{\beta_i}([0,1]^d)$ for
$i \in \{1,\dots,d\}$. If
\begin{equation}
\label{eq:ZustYoung}
\gamma + \beta_1 + \cdots + \beta_d > d,
\end{equation}
then the integral
\[
\int_{[0,1]^d} f \, dg_1 \wedge \cdots \wedge dg_d
\]
is well defined as a certain limit of Riemann-Stieltjes sums.
\end{thm}

The classical Young integral~\cite{Youn} is recovered in the case \(d=1\).
In recent years, this integral has attracted significant attention.
For instance, \cite{AlbeStepTrev} provides an alternative construction of the two-dimensional Züst integral,
while \cite{Jaffard2025Holder} investigates the Hölder regularity of Hölder differential forms arising in the Züst integral.
Since Züst's original work, most contributions have focused on this integral,
largely due to its connections with the Gromov--Hölder problem for the Heisenberg group;
see, for example, \cite{HajlaszMirraSchikorra2025HoelderHeisenberg, BaloghKozhevnikovPansu}.

Condition~\eqref{eq:ZustYoung} is known to be optimal (improving upon it falls within the scope of rough ``sheet'' theory). From one point of
view, this is somewhat unsatisfactory, as the regularity requirements
become increasingly stringent as the dimension $d$ grows. It is
precisely this restrictive Young regime, however, that we seek to investigate
more closely in the present work.

In Section~4, we consider Young integration over (generalized) $m$-dimensional domains in $[0,1]^d$, building on two existing frameworks: Whitney's geometric integration theory, and the more recent theory of Moonens, De Pauw, and Pfeffer, which establishes a duality between normal currents with charges in middle dimension. The latter provides a functional-analytic setting for non-absolute integration in nonzero codimension, and allows the development of an exterior calculus for non-smooth differential forms, generalizing the Züst integral.

We have not attempted to state results in their utmost generality. For instance, most results are formulated on the cube $[0, 1]^d$, even though this restriction is often inessential and can be removed at the price of additional technicalities. More generally, several arguments are presented as outlines rather than fully optimized proofs. Our intention has been to emphasize the main ideas and we provide accurate references to the literature where complete and more general treatments can be found.

This paper grew out of a talk given at the 2025 Summer School
\emph{Maximal Functions \& Differentiation of Integrals in}
$\R^n$, held at Westlake University in Hangzhou. Beyond classical
material, most of the results presented here are drawn from~\cite{BouaDePa, Bouafia2025Young, Boua2}, and no new results are claimed.
The author warmly thanks the organizers for the invitation and for their
generous hospitality, and is especially grateful to Thierry De Pauw for his careful proofreading.

\input{nonabsint}

\input{strongcharges}

\input{geomint}

\begin{appendix}
\section{Notation index}
\begin{center}
\begin{longtable}{p{.155\textwidth}p{.63\textwidth}p{.13\textwidth}}
\toprule
Object & Meaning & Ref. \\
\midrule
\endhead
\bottomrule
\endfoot
$A \ominus B$ & Symmetric difference between sets $A$ and $B$ & \\
$|B|$ & Lebesgue (outer) measure of $B \subset \R^d$ & \\
$\|B\|$ & perimeter of $B$ & Def~\ref{def:deGiorgiVar} \\
$\scrH^m$ & $m$-dimensional (outer) Hausdorff measure & \\
$ (\mathrm{HK}) \int$ & Henstock-Kurzweil integral & Def~\ref{def:HK}\\
$ (\mathrm{H}) \int$ & Henstock integral & Def~\ref{def:Hintegral}\\
$ (\mathrm{Pf}) \int$ & Pfeffer or Pfeffer-Stieltjes integral & Def~\ref{def:pfef}, Def~\ref{def:pfefS}\\
$ (\mathrm{Y}) \int$ & Young integral (with respect to a fractional charge) & Thm~\ref{thm:young}\\
$C_0[0, 1]$ & space of continuous functions on $[0, 1]$ vanishing at~$0$ & Subsec~\ref{subsec:topHK} \\
$\Delta_f(I)$ & rectangular increment of $f$ on $I$ (rectangle or rectangular figure) & Eq~\eqref{eq:rectIncr}, Subsec~\ref{subsec:incr}\\
$HK[0,1]$ & space of Henstock-Kurzweil integrable functions & Subsec~\ref{subsec:topHK} \\
$H([0, 1]^d)$ & space of Henstock integrable functions on $[0, 1]^d$ & Subsec~\ref{subsec:Henstock} \\
$\mathrm{Pf}([0, 1]^d)$ & space of Pfeffer integrable functions on $[0, 1]^d$ & Subsec~\ref{subsec:pfef} \\
$BV[0, 1]$ & space of functions of bounded essential variation on $[0, 1]$ & Def~\ref{def:variation}\\
$BV([0, 1]^d)$ & space of functions $\R^d \to 0$ of bounded variation, vanishing a.e. outside $[0, 1]^d$ & Subsec~\ref{subsec:Ch} \\
$BV^\infty([0, 1]^d)$ & $BV([0, 1]^d) \cap L^\infty([0, 1]^d)$ & Subsec~\ref{subsec:Ch}\\
$\niceBV([0, 1]^d)$ & algebra of $BV$-subsets of $[0, 1]^d$ & Subsec~\ref{subsec:Ch} \\
$\mathrm{V}(g)$ & Vitali variation of a function $g \colon [0, 1]^d \to \R$ & Def~\ref{def:variation}, Eq~\eqref{eq:vitali} \\
$Du$ & distributional gradient measure of a function $u$ of bounded variation & Def~\ref{def:deGiorgiVar} \\
$\|Du\|$ & total variation measure of $Du$ & Def~\ref{def:deGiorgiVar}\\
$\partial T$ & boundary of a flat $m$-chain $T$ & Subsec~\ref{subsec:whitney}\\
$\bM(T)$ & mass of $T$ & Subsec~\ref{subsec:whitney}\\
$\|T\|$ & density measure associated to a current $T$ of finite mass & Eq~\eqref{eq:massfinie}\\
$\vec{T}$ & $m$-vectorfield associated to an $m$-current $T$ of finite mass & Eq~\eqref{eq:massfinie}\\
$\bF(T)$ & flat norm of $T$ & Subsec~\ref{subsec:whitney}\\
$\bN(T)$ & normal mass of $T$ & Subsec~\ref{subsec:normal} \\
$\bP_m([0, 1]^d)$ & space of polyhedral $m$-chains & Subsec~\ref{subsec:whitney}\\
$\bF_m([0, 1]^d)$ & space of flat $m$-chains & Def~\ref{def:flat}\\
$\bN_m([0, 1]^d)$ & space of normal $m$-currents & Subsec~\ref{subsec:normal}\\
$\bF_m^\gamma([0, 1]^d)$ & space of $\gamma$-fractional $m$-currents & Def~\ref{def:fractCurrent} \\
$\bF^m([0, 1]^d)$ & space of flat $m$-cochains & Def~\ref{def:flat} \\
$\bCH^m([0, 1]^d)$ & space of $m$-charges & Def~\ref{def:chm}\\
$\bCH^{m, \gamma}([0, 1]^d)$ & space of $\gamma$-fractional $m$-charges & Def~\ref{def:chgamma}\\
\end{longtable}
\end{center}

\end{appendix}

\bibliographystyle{amsalpha} 

\bibliography{phil.bib}

\end{document}

%% file: nonabsint.tex
\section{Non-absolute integrals and the space of (strong) charges}

\subsection{Henstock-Kurzweil integration} An integration theory is called \emph{non-absolute} when integrability does not imply integrability of the absolute value, so that convergence can rely on cancellation effects. In contrast, the Lebesgue integral is an absolute theory: a function $f$ is Lebesgue-integrable if and only if $|f|$ is. One motivation for considering non-absolute integrals is that the Lebesgue integral fails to integrate every derivative. Indeed, one easily establishes that the derivative of
\[
x \in [0, 1] \mapsto \begin{cases}
 x^2 \sin(1/x^2) & \text{if } x \neq 0 \\
 0 & \text{if } x = 0
\end{cases}
\]
is not in $L^1([0, 1])$. The Henstock-Kurzweil integral is the archetypal and best understood example of a non-absolute integral; it integrates every derivative and strictly extends the Lebesgue integral. It arises from a modest but effective improvement of the Riemann integral, in which the uniform control of partition sizes is replaced by a local, pointwise condition governed by gauges.

Let us make precise statements. A \emph{tagged partition} $\scrP$ of $[0, 1]$ is a finite subdivision $0 = a_0 < a_1 < \cdots < a_n = 1$, together with a choice of tags $x_i \in [a_i, a_{i+1}]$ for each $i \in \{0, 1, \dots, n-1\}$. Given a function $f \colon [0, 1] \to \R$, its \emph{Riemann sum} over $\scrP$ is simply
\[
\scrS(f, \scrP) = \sum_{i=0}^{n-1} f(x_i) (a_{i+1} - a_i)
\]
A \emph{gauge} (over $[0, 1]$) is any positive function $\delta \colon [0, 1] \to \R^{> 0}$ with no further assumption---measurability is not even required.
The tagged partition $\scrP$ is said to be \emph{$\delta$-fine} whenever $\diam [a_i, a_{i+1}] < \delta(x_i)$ for all $i \in \{0, 1, \dots, n-1\}$.

\begin{defn}
\label{def:HK}
  A function $f \colon [0, 1] \to \R$ is said to be \emph{Henstock-Kurzweil integrable} (with integral $I$) whenever for all $\varepsilon > 0$, there exists a gauge $\delta$ such that
  \[
  \left| \scrS(f, \scrP) - I \right| \leq \varepsilon
  \]
  for all $\delta$-fine tagged partitions $\scrP$ of $[0, 1]$. We denote 
  \[
  I = (\mathrm{HK}) \int_0^1 f(x) \, dx.
  \]
\end{defn}

The well-definedness of this notion relies on the classical Cousin lemma, which asserts that for every gauge $\delta$ over $[0, 1]$ there exists at least one $\delta$-fine tagged partition. This elementary result of real analysis is in fact equivalent to the compactness of the interval $[0, 1]$. In particular, it ensures that the Henstock-Kurzweil integral $(\mathrm{HK}) \int_0^1 f(x) \, dx$ is uniquely determined.

\begin{example}
  Let $N \subset [0, 1]$ be a Lebesgue null set. We illustrate how the introduction of gauges makes it possible to integrate the indicator function $\ind_{N}$, which is in general not Riemann-integrable. Fix $\varepsilon > 0$. There is an open set $U$ that contains $N$ and such that $|U| \leq \varepsilon$ ($|U|$ stands for the Lebesgue measure of $U$). For any $x \in N$, we set $\delta(x) = \dist(x, \R \setminus U)$, while for $x \in [0, 1] \setminus N$, we choose $\delta(x) > 0$ arbitrarily. It is then clear that $|\scrS(f, \scrP)| \leq |U| \leq \varepsilon$ for any $\delta$-fine tagged partition $\scrP$ of $[0, 1]$.
  
  More generally, all Lebesgue-integrable functions are Henstock-Kurzweil integrable.
\end{example}

\begin{thm}
  \label{thm:fundHK}
   Let $f \colon [0, 1] \to \R$ be a differentiable function. Then $f'$ is Henstock-Kurzweil integrable and
   \[
   (\mathrm{HK}) \int_0^1 f'(x) \, dx = f(1) - f(0).
   \]
\end{thm}

\begin{proof}
  Fix $\varepsilon > 0$. For any $x \in [0, 1]$, there is $\delta(x) > 0$ such that
  \[
  \left| f(b) - f(a) - f'(x)(b - a) \right| \leq \varepsilon (b-a)
  \]
  for any compact interval $[a, b] \subset [0, 1]$ that contains $x$ and such that $\diam [a, b] < \delta(x)$. By a telescoping argument, it follows that
  \[
  \left| \scrS(f', \scrP) - (f(1) - f(0)) \right| \leq \varepsilon
  \]
  for all $\delta$-fine partitions $\scrP$.
\end{proof}

\subsection{A topology on $HK[0, 1]$}
\label{subsec:topHK}

The notion of a charge naturally appears when one attempts to topologize the space $HK[0,1]$ of Henstock-Kurzweil integrable functions. While there is a canonical choice of topology, it is not entirely straightforward due to the non-absolute nature of the integral. Here, we shall be satisfied with equipping $HK[0,1]$ with a topology under which the topological dual $HK[0,1]^*$ admits a convenient and explicit representation.

Motivated by the classical duality between $L^1([0,1])$ and $L^\infty([0,1])$, a natural candidate for the dual of $HK[0,1]$ is the space of multipliers.
By a \emph{multiplier} for the Henstock-Kurzweil integral (over $[0, 1]$), we mean a function $g \colon [0, 1] \to \R$ such that 
\[
fg \in HK[0, 1] \text{ for all } f \in HK[0, 1].
\]
\begin{rmrk}
Multipliers can also be considered in the context of other notions of integration. 
For example, it is a straightforward exercise to characterize the multipliers of conditionally convergent series: 
a sequence $(v_n)$ has the property that the series $\sum_{n=0}^\infty u_n v_n$ converges for every convergent series $\sum_{n=0}^\infty u_n$ 
if and only if it is of bounded variation, \ie{}, $\sum_{n=0}^\infty |v_{n+1} - v_n| < \infty$.
\end{rmrk}
We now characterize the multipliers for the Henstock-Kurzweil integral.
\begin{defn}
\label{def:variation}
  The \emph{variation} of a function $g \colon [0, 1] \to \R$ is the quantity
  \[
  \mathrm{V}(g) = \sup \left\{ \sum_{i=0}^{n-1} |g(a_{i+1}) - g(a_i)| : 0 = a_0 < a_1 < \cdots < a_n = 1 \right\}.
  \]
  In case $\mathrm{V}(g) < \infty$, we say that $g$ is \emph{of bounded variation}. We say that $g$ is \emph{of essential bounded variation} whenever there is a function $\bar{g}$ such that $g = \bar{g}$ almost everywhere and $\mathrm{V}(\bar{g}) < \infty$. We denote by $BV[0, 1]$ the space of functions of essential bounded variation, equipped with the norm
  \[
  \|g\|_{BV[0, 1]} = \| g \|_\infty + \inf \left\{ \mathrm{V}(\bar{g}) : g = \bar{g} \text{ a.e} \right\}.
  \]
\end{defn}

\begin{thm}[{\cite[Theorem~6.1.9]{Yeong}}]
\label{thm:dualHK}
  The multiplier space for $HK[0, 1]$ is $BV[0, 1]$.
\end{thm}

Theorem~\ref{thm:dualHK} tells us that the natural domains of integration for the
Henstock-Kurzweil integral are those sets $B \subset [0,1]$ whose indicator function
belongs to $BV[0,1]$, that is, finite unions of compact intervals, up to Lebesgue null sets.

By associating to each function $f \in HK[0, 1]$ its indefinite integral 
\[
F \colon x \in [0, 1] \mapsto (\mathrm{HK}) \int_0^x f,
\] we can naturally view $HK[0, 1]$ as a subspace of $C_0[0, 1]$, the space of continuous functions vanishing at $0$, with the supremum norm. The continuity of $F$ is indeed proven in \cite[Theorem~9.12]{Gordon1994}. The \emph{Alexiewicz norm} of $f$ is defined by
\[
\|f\|_{HK[0, 1]} = \max \left\{ \left| (\mathrm{HK}) \int_0^x f \right| : x \in [0, 1] \right\}.
\]
At first glance, this choice of norm may seem counterintuitive---we are, in a sense, identifying a function with its indefinite integral. Yet it leads to the following duality result.

\begin{thm}[Alexiewicz, \cite{Alex}]
   The continuous linear maps on $HK[0, 1]$ are of the form
   \[
   f \mapsto (\mathrm{HK}) \int_0^1 fg,
   \]
   where $g \in BV[0, 1]$. This results in an isomorphism $HK[0, 1]^* \simeq BV[0, 1]$.
\end{thm}

One drawback of the Alexiewicz norm is that the space $HK[0, 1]$ is not complete, it is only isometric to a dense subspace of $C_0[0, 1]$. It is shown however in \cite{Osta} that the Alexiewicz topology is the strongest locally convex topology compatible with the duality between $HK[0, 1]$ and $BV[0, 1]$, and the only normable one. We also note the work \cite{Kurz}, where an alternative vector space topology is constructed for $HK[0, 1]$, that is complete but not locally convex.

Let us mention in passing that another motivation for introducing a topology on
$HK[0,1]$ is to investigate its relationship with convergence theorems.
Convergence results for the Henstock--Kurzweil integral do exist; for instance,
it is straightforward to show that
\[
\lim_{n\to\infty} (\mathrm{HK}) \int_0^1 f_n
=
(\mathrm{HK}) \int_0^1 \lim_{n\to\infty} f_n
\]
whenever $(f_n)$ is a sequence of equi-integrable functions in $HK[0, 1]$ converging
pointwise. Here, equi-integrability means that, given $\varepsilon>0$, a single
gauge can be chosen in the definition of Henstock--Kurzweil integrability
uniformly with respect to $n$.

However, such convergence theorems lack the practicality of the Lebesgue
dominated convergence theorem, and are not proven to be ``optimal''. In the author's opinion, they remain
poorly understood. This difficulty is closely related to the fact that it is
hard to formulate useful (weak) compactness criteria in a non-complete space such as
$HK[0,1]$ (see \cite{BongiornoPanchapagesan1995} for results in this direction). For example, the following question is open.

\begin{quest}
  Find a convergence theorem for the Henstock-Kurzweil integral that rigorously justifies the following formal proof of the fundamental theorem of calculus. Let $f \colon \R \to \R$ be a differentiable function. One has
  \begin{align*}
  (\mathrm{HK}) \int_0^1 f'(x) \, dx & = (\mathrm{HK}) \int_0^1 \lim_{n \to \infty} n \left(f\left(x + \frac{1}{n}\right) - f(x) \right) dx \\
                                    & = \lim_{n \to \infty} (\mathrm{HK}) \int_0^1 n \left(f\left(x + \frac{1}{n}\right) - f(x) \right) dx \\
                                   &  = \lim_{n \to \infty} \left( \frac{1}{n} \int_{1}^{1 +1/n} f(x) \,  dx - \frac{1}{n} \int_0^{1/n} f(x) \, dx \right) \\
                                    & = f(1) - f(0).
  \end{align*}
\end{quest}

\subsection{Henstock integral in higher dimension}
\label{subsec:Henstock}

When attempting to define a gauge integral of Henstock-Kurzweil type in dimension $d$, one major obstacle is to identify the appropriate generalization of intervals. The most straightforward gauge integral for functions on $[0, 1]^d$ is the Henstock integral, that closely mirrors the Henstock-Kurzweil integral in dimension one, with intervals replaced by non-overlapping tagged rectangles.

\begin{defn}
\label{def:Hintegral} A \emph{rectangle} is a set of the form $\prod_{i=1}^d [a_i, b_i]$, where $a_i \leq b_i$ for each $i \in \{1, \dots, d\}$.
  Two rectangles $I$ and $I'$ are \emph{non-overlapping} whenever $|I \cap I'| = 0$. A \emph{tagged partition} (into rectangles) of $[0, 1]^d$ is a finite set of the form 
  $\scrP = \{(I_i, x_i) : i \in \{1, \dots, n \}\}$ where the $I_1, \dots, I_n$ are pairwise non-overlapping rectangles such that $[0, 1]^d = \bigcup_{i=1}^n I_i$ and $x_i \in I_i$ for each $i$. Given a function $f \colon [0, 1]^d \to \R$, its \emph{Riemann sum} over $[0, 1]^d$ is
  \[
\scrS(f, \scrP ) = \sum_{i=1}^n f(x_i) |I_i|. 
  \]
   A \emph{gauge} is a positive function $\delta \colon [0, 1]^d \to \R^{> 0}$. The tagged partition $\scrP$ is said to be \emph{$\delta$-fine} whenever $\diam I_i < \delta(x_i)$ for all $i$.
   
   A function $f \colon [0, 1]^d \to \R$ is said to be \emph{Henstock integrable} (with integral $I$) whenever for all $\varepsilon > 0$, there exists a gauge $\delta$ such that
  \[
  \left| \scrS(f, \scrP) - I \right| \leq \varepsilon
  \]
  for all $\delta$-fine tagged partitions $\scrP$ of $[0, 1]^d$.
\end{defn}

A natural question is whether one can formulate a divergence theorem for the Henstock integral that is valid for every differentiable vector field. Such a result would constitute a higher-dimensional analog of the fundamental theorem of calculus. Unfortunately, the answer is negative. This non-absolute integral does, however, satisfy a Fubini theorem (see \cite[Section~6.1]{McLeod1980}, and a remarkable argument due to W. Pfeffer shows that a gauge integral cannot simultaneously satisfy both a generalized Fubini theorem and divergence theorem, see \cite[Example~11.1.2]{Pfef}.

We now describe the multipliers associated with the Henstock integral. They form a space of functions of bounded variation, though not in the modern sense of the term. To make this precise, we introduce some terminology.

Given a function $f \colon [0, 1]^d \to \R$ and a rectangle $I = \prod_{i=1}^d [a_i, b_i] \subset [0, 1]^d$, the \emph{rectangular increment} of $f$ over $I$ is the quantity
\begin{equation}
\label{eq:rectIncr}
\Delta_f (I) = \sum_{(c_i) \in \prod_{i=1}^d \{a_i, b_i\}} (-1)^{\delta_{a_1, c_1}} \cdots (-1)^{\delta_{a_d, c_d}} f(c_1, \dots, c_d).
\end{equation}
In particular, in dimension~2, one recovers the perhaps better-known rectangular increment $\Delta_f(I) = f(b_1, b_2) - f(a_1, b_2) - f(b_1, a_2) + f(a_1, b_1)$. It is worth noting that, if $f$ is of class $C^d$, then $\Delta_f(I)$ is just
\[
\Delta_f(I) = \int_{I} \frac{\partial^d f}{\partial x_1 \cdots \partial x_d} (x) \, dx.
\]
The \emph{Vitali variation} of a function $f$ is defined as
\begin{equation}
\label{eq:vitali}
\mathrm{V}(f) = \sup \sum_{I \in \scrP} |\Delta_f(I)|
\end{equation}
where the supremum is taken over all (untagged) partitions $\scrP$ of $[0, 1]^d$ into pairwise non-overlapping rectangles.

If \( f = \ind_B \) is an indicator function, then the Vitali variation does not, in general, coincide with the perimeter of \( B \), contrary to what one might expect. For example, let \( B \subset [0,1]^2 \) be the ``rotated'' square with vertices \( (1/2,0), (1,1/2), (1/2,1), (0,1/2) \). Then the Vitali variation of \( \mathbf{1}_B \) is infinite. By contrast, if \( B \) is a rectangle (with sides parallel to the coordinate axes, as in Definition~\ref{def:Hintegral}) the Vitali variation equals \(4\), the number of its vertices!

\begin{thm}[{\cite[Theorem~5.1]{LeeChewLee1996}}]
  The multiplier space for the $d$-dimensional Henstock integral consists exactly of those functions $[0, 1]^d \to \R$ that coincide almost everywhere with functions $g$ such that, for all $J \subsetneq \{1, \dots, d\}$ and for all $(x_i)_{i \in J}$, the partial map
  \[
  (x_i)_{i \not\in J} \mapsto g(x_1, \dots, x_d)
  \]
  defined over $[0, 1]^{d - \# J}$, has finite Vitali variation.
\end{thm}

The preceding discussion shows that the multiplier space is not invariant under rotations; as a consequence, the space \( H([0,1]^d) \) of Henstock integrable functions also fails to be rotation invariant as well. More generally, the entire approach based on rectangles makes it difficult to derive a practical change-of-variables formula for the Henstock integral.

As in the one-dimensional case, it is possible to realize the space of multipliers as the dual of $H([0, 1]^d)$, provided that we equip it with the $d$-dimensional Alexiewicz norm:
\[
\|f\|_{H([0, 1]^d)} = \max \left\{ \left| (\mathrm{H}) \int_{[0, x_1] \times \cdots \times [0, x_d]} f\right| : x = (x_1, \dots, x_d) \in [0, 1]^d \right\}.
\]
Then $H([0, 1])^d$ can be isometrically identified with a strict dense subspace of $C_0([0, 1]^d)$, the  space of continuous functions that vanish on the coordinate hyperplanes $\bigcup_{i=1}^d \{x \in [0, 1]^d : x_i = 0\}$, with the supremum norm. In particular, it is not complete. 

\subsection{Charges}
\label{subsec:Ch}

We aim to define a non-absolute integral in dimension 
$d$ whose multiplier space consists of functions of bounded variation, in the modern sense of the term. To achieve this, we must abandon the rectangle-based approach and revise the notion of an indefinite integral. Whereas, for the Henstock integral, the appropriate notion of an indefinite integral for $f \in H([0, 1]^d)$ was the map
\begin{equation}
\label{eq:primitiveRect}
x \mapsto (\mathrm{H}) \int_{[0, x_1] \times \cdots \times [0, x_d]} f
\end{equation}
defined by integration over rectangles, we shall instead consider indefinite integrals as set functions. These assign to each suitable domain $B$ the value $\int_B f$. For the Pfeffer integral, the domains of integration are bounded sets of finite perimeter. We therefore begin by recalling the definitions of functions of bounded variation and of sets of finite perimeter.

\begin{defn}
\label{def:deGiorgiVar}
  Let $u \in L^1_{\mathrm{loc}}(\R^d)$ be a map. Its \emph{variation} (in the sense of De Giorgi) is the quantity
  \[
  \|Du\|(\R^d) = \sup \left\{ \int_{\R^d} u \diver v : v \in C^1_c(\R^d; \R^d) \text{ and } \|v\|_\infty \leq 1 \right\}.
  \]
  In case $\|Du\|(\R^d) < \infty$, then we say that $u$ is \emph{of bounded variation}. By the Riesz representation theorem, it follows that the distributional gradient of $u$ is finite $\R^d$-measure $Du$. Its total variation measure is denoted by $\|Du\|$. A measurable set $B$ is said to be \emph{of finite perimeter} whenever $\ind_B$ is of finite variation. The \emph{perimeter} of $B$ is the quantity $\|B\| = \|D \ind_B \|(\R^d)$.
\end{defn}

In the sequel, we shall define the Pfeffer integral on the domain $[0,1]^d$. This choice of domain is not essential---in contrast to Subsection~\ref{subsec:Henstock}---and is made solely for the sake of simplicity and notational convenience (and because some arguments in the later sections will work better over the cube $[0, 1]^d$). We say that a measurable set $B \subset \R^d$ is a \emph{$BV$-set} if it is bounded and of finite perimeter. The collection of $BV$-subsets of $[0, 1]^d$ will be denoted $\niceBV([0, 1]^d)$.

We let $BV([0, 1]^d)$ be the set of maps $u \in L^1(\R^d)$ such that $u = 0$ a.e on $\R^d \setminus [0, 1]^d$ and $\|Du\|(\R^d) < \infty$, normed by
\[
\|u\|_{BV} = \|Du\|(\R^d).
\]
It would be also possible to norm $BV([0, 1]^d)$ with the equivalent norm $\|u\|_{L^1} + \|Du\|(\R^d)$ but the former choice is more appropriate as it makes $BV([0, 1]^d)$ isometric to a dual space, see Theorem~\ref{thm:dualCH}.
From dimension $d \geq 2$, functions of bounded variation are not necessarily bounded. We also consider the algebra $BV^\infty([0, 1]^d) = BV([0, 1]^d) \cap L^\infty(\R^d)$ normed by
\[
\|u\|_{BV^\infty} = \|u\|_\infty + \|Du\|(\R^d).
\]

\begin{defn}
\label{def:charge}
  A \emph{charge} over $[0, 1]^d$ is a set function $\omega \colon \niceBV([0, 1]^d) \to \R$ satisfying the following properties:
  \begin{itemize}
  \item \emph{additivity}: $\omega(B_1 \cup B_2) = \omega(B_1) + \omega(B_2)$ if $B_1, B_2 \in \niceBV([0, 1]^d)$ are disjoint;
  \item \emph{continuity}: $\omega(B_n) \to \omega(B)$ for any sequence $(B_n)$ in $\niceBV([0, 1]^d)$ such that $|B_n \ominus B| \to 0$ and $\sup_n \|B_n\| < \infty$ (here, $\ominus$ denotes symmetric difference).
  \end{itemize}
\end{defn}
The first appearance of charges in the study of non-absolute integrals can be traced back to Ma\v{r}\'ik \cite{Marik1965Extensions}.
The condition $\sup_n \|B_n\| < \infty$ in the continuity statement ensures that the $B_n$ do not fractalize excessively. It is inspired by continuity properties of the Henstock-Kurzweil integral (see for instance Hake's theorem \cite[Theorem~6.1.6]{Pfef}). Note also that it implies that $\omega(B) = 0$ whenever $|B| = 0$.

\begin{example}[Case $d = 1$]
\label{ex:chargedim1}
  A $BV$-set in $[0, 1]$ is, up to a Lebesgue-null set, a finite union of compact intervals. The function $x \in [0, 1] \mapsto \omega([0, x])$ determines the charge $\omega$ by additivity. From the continuity property of charges, it belongs to $C_0[0, 1]$. Reciprocally, any function $v \in C_0[0, 1]$ determines the charge
  \[
  dv \colon \bigcup_{i=1}^n [a_i, b_i] \mapsto \sum_{i=1}^n (v(b_i) - v(a_i)).
  \]
\end{example}

We can thus identify the space of charges on $[0,1]$ with $C_0[0,1]$, the completion of $HK[0,1]$ under the Alexiewicz norm. More generally, the space of charges is designed to provide a natural home for indefinite integrals arising from non-absolute integrals, such as the Pfeffer integral, which will play a central role in its construction.

\begin{example}[Indefinite integral of $L^1$ functions]
\label{ex:chargeL1}
  Let $f \in L^1([0, 1]^d)$. Then the map $T_f \colon B \mapsto \int_B f$ is a charge. The additivity of $T_f$ is clear. As for continuity $T_f(B_n) \to T_f(B)$ as soon as $|B_n \ominus B| \to 0$.
\end{example}

For the next example, we need to know a little more about the structure of sets of finite perimeter. Each such set $B$ admits a so-called \emph{reduced boundary} $\partial^* B$, defined in a measure-theoretic way, that is $(d-1)$-rectifiable, and a normal outer vector field $n_B$ defined $\scrH^{d-1}$-almost everywhere on $\partial^*B$, such that
\[
\int_B \diver v = \int_{\partial^* B} v \cdot n_B \, d\scrH^{d-1}
\]
for all $v \in C^1_c(\R^d; \R^d)$, see \cite[Theorem~5.15]{EvanGari}.

\begin{example}[Strong charge]
  Let $v \in C([0, 1]^d; \R^d)$. The \emph{flux} of $v$ is
  \[
  \diver v \colon B \in \niceBV([0, 1]^d) \mapsto \int_{\partial^* B} v \cdot n_B \, d\scrH^{d-1}
  \]
 The additivity of $\diver v$ is clear if $v \in C_c^1([0, 1]^d; \R^d)$ and can be obtained in the general case by an approximation argument. As for continuity, let $(B_n)$ be a sequence in $\niceBV([0, 1]^d)$ such that $|B_n| \to 0$ and $\sup \|B_n\| < \infty$. Fix $\varepsilon > 0$ and $w \in C^1([0, 1]^d; \R^d)$ such that $\|v - w\|_\infty \leq \varepsilon$. Then
  \[
  |\diver v(B_n)| \leq \int_{\partial^* B_n} |v - w| \, d\scrH^{d-1} + \left|\int_{B_n} \diver w \right|.
  \]
  The first term is bounded by $\varepsilon \sup_n \|B_n\|$, and the second one tends to $0$ because $\diver w \in L^1([0, 1]^d)$. By the arbitrariness of $\varepsilon$, this shows that $\diver v(B_n) \to 0$. For historical reasons, charges of the form $\diver v$ are called \emph{strong}.
\end{example}

In fact, those two examples exhaust all possibilities. It can be proven that any charge can be decomposed as a sum $T_f + \diver v$. For functional-analytical applications, it might be more appropriate to see charges as linear functionals on $BV^\infty([0, 1]^d)$, continuous with respect to some weak* topology. It is possible by the following result, that gives in particular a Banach space structure to the space of charges.

\begin{thm}[{\cite{BuczDePaPfef}}]
  For any charge $\omega$, there is a unique $\bar{\omega} \in BV^\infty([0, 1]^d)^*$ such that
  \begin{equation}
  \label{eq:extensionBVinfty}
  \bar{\omega}(u_n) \to 0 \text{ if } \|u_n\|_{L^1} \to 0 \text{ and } \sup_n \|Du_n\|(\R^d) < \infty
  \end{equation}
  and $\omega(B) = \bar{\omega}(\ind_B)$ for all $B \in \niceBV([0, 1]^d)$. It is given by
  \[
  \bar{\omega}(u) = \int_0^\infty \omega(\{u > t \}) \, dt
  \]
  for all nonnegative $u \in BV^\infty([0, 1]^d)$ (we recall that $\{u > t\} \in \niceBV([0, 1]^d)$ for almost every $t > 0$, as a consequence of the coarea formula \cite[Section~5.5]{EvanGari}), and $\bar{\omega}(u) = \bar{\omega}(u^+) - \bar{\omega}(u^-)$ in the general case.
  
  Conversely, any functional $\bar{\omega} \in BV^\infty([0, 1]^d)^*$ that satisfies the continuity condition~\eqref{eq:extensionBVinfty} defines a charge $\omega \colon B \mapsto \bar{\omega}(\ind_B)$.
\end{thm}

\subsection{A differentiation result}

Suppose we aim to establish a divergence theorem for differentiable vector fields for some non-absolute integration method. If we attempt to replicate the proof of Theorem~\ref{thm:fundHK} in dimension~\( d \), we encounter the need, at a certain stage, to ensure that
\[
\left|\int_{\partial^* B} v \cdot n_B \, d \scrH^{d-1}
 - \diver v(x)\,|B|\right|
\leq \varepsilon\,|B|,
\]
where \( v \) is a differentiable vector field and \( B \) is a $BV$-set around some point \( x \), assumed to be sufficiently small (depending on $\varepsilon$ and $x$). The following proposition shows that achieving such an estimate requires quantitative control on the regularity of \( B \), measured by
\[
\reg B = \frac{|B|}{\|B\| \diam B}.
\]

\begin{prop}
\label{prop:fluxdens}
  Let $v \colon [0, 1]^d \to \R^d$ be a continuous vector field, differentiable at $x \in [0, 1]^d$. For any sequence $(B_n)$ of non-negligible $BV$-subsets of $[0, 1]^d$ that shrinks to $x$ in the following sense: 
  \[
  x \in B_n, \quad \diam B_n \to 0, \quad \inf_n \reg B_n > 0
  \]
  we have
  \[
  \frac{1}{|B_n|} \int_{\partial^* B_n} v \cdot n_{B_n} \, d \scrH^{d-1} \to (\diver v)(x).
  \]
\end{prop}

\begin{proof}
  By assumption, there is a linear map $L \colon \R^d \to \R^d$ so that for all $\varepsilon > 0$, there is $\delta$ such that
  \[
  \|v(y) - v(x) - L(y - x)\| \leq \varepsilon \|y - x\| \text{ whenever } \|x - y\| \leq \delta.
  \]
  Suppose $n$ is large enough so that $\diam B_n \leq \delta$. Then
  \[
  \left| \int_{\partial^* B_n} (v(y) - v(x) - L(y - x)) \cdot n_{B_n}(y) \, d \scrH^{d-1}(y) \right| \leq \varepsilon  \| B_n \| \diam B_n.
  \]
  However,
  \begin{gather*}
  \int_{\partial^* B_n} v(x) \cdot n_{B_n}(y) \, d \scrH^{d-1}(y) = 0 \\
   \int_{\partial^* B_n} L(y-x) \cdot n_{B_n}(y) \, d\scrH^{d-1}(y) = \int_{B_n} \diver L = |B_n| \diver v(x).
  \end{gather*}
  which readily implies that
  \[
  \left| \frac{1}{|B_n|} \int_{\partial^* B_n} v \cdot n_{B_n} \, d \scrH^{d-1} - \diver v(x) \right| \leq \frac{\varepsilon}{\reg B_n}.
  \]
  The conclusion follows.
\end{proof}

The introduction of a regularity coefficient goes back to J.~Mawhin~\cite{Mawh}, who used it to construct the first non-absolute integral satisfying a divergence theorem. As a simple example of such an integral, one may consider the \emph{dyadic Henstock integral}. It is obtained by modifying the definition of the Henstock integral on \([0,1]^d\) by restricting to tagged partitions consisting only of dyadic cubes. A dyadic version of the Cousin lemma holds: using a simple dichotomy argument, one can construct, for any gauge $\delta$, a $\delta$-fine partition of $[0,1]^d$ into dyadic cubes. Since the regularity coefficient of a cube is constant, it follows readily---using Proposition~\ref{prop:fluxdens}---that the divergence theorem holds for any differentiable vector field on \([0,1]^d\).

\subsection{Pfeffer integral}
\label{subsec:pfef}

The Pfeffer integral resolves several of the difficulties associated with the higher-dimensional Henstock integral. It is a gauge integral defined using 
$BV$-sets rather than rectangles. Before anything else, it may be helpful to clarify what we mean by the ``Pfeffer integral,'' since it is known under several names in the literature. In particular, it is called the variational integral in \cite{Pfef2} and the $R$-integral in \cite{Pfef3}. We also refer to its treatment in \cite[Chapter 48]{Fremlin2003MeasureTheoryV4} when the domain of integration is $\mathbb{R}^d$.

Before introducing the relevant definitions, we present an alternative formulation of the one-dimensional Henstock-Kurzweil integral that avoids the notion of partitions and is closer in spirit to the definition of the Pfeffer integral.

A \emph{tagged division} of $[0, 1]$ is a finite set of the form $\scrD = \{([a_i, b_i], x_i) : i \in \{1, \dots, n\}\}$, where the $[a_i, b_i]$ are pairwise non-overlapping compact subintervals of $[0, 1]$ and $x_i \in [a_i, b_i]$ for each $i \in \{1, \dots, n\}$. We do not require $[0, 1] = \bigcup_{i=1}^n [a_i, b_i]$. Given a gauge $\delta$ on $[0, 1]$, the division $\scrD$ is said to be \emph{$\delta$-fine} whenever $\diam [a_i, b_i] < \delta(x_i)$ for each $i$. 

\begin{lem}[Saks-Henstock Lemma]
  A function $f \colon [0, 1] \to \R$ is Henstock-Kurzweil integrable if and only if there is a function $g \in C_0[0, 1]$ such that, for all $\varepsilon > 0$, we can find a gauge $\delta$ such that, for each $\delta$-fine division $\scrD  = \{([a_i, b_i], x_i) : i \in \{1, \dots, n\}\}$ of $[0, 1]$, one has
  \[
  	\sum_{i=1}^n \left| f(x_i) (b_i - a_i) - (g(b_i) - g(a_i)) \right| \leq \varepsilon.
  \]
  In this case, $(\mathrm{HK}) \int_0^1 f = g(1)$.
\end{lem}

\begin{proof}
 The indirect sense ($\impliedby$) is easier. Let $\varepsilon > 0$ and $\delta$ be a gauge chosen in accordance with $\varepsilon$. For any $\delta$-fine partition $\scrP = \{ ([a_i, a_{i+1}], x_i) : i \in \{0, \dots, n-1\}\}$, one has
 \[
 \left| \scrS(f, \scrP) - g(1) \right| \leq \sum_{i=0}^{n-1} \left| f(x_i) (a_{i+1} - a_i) - (g(a_{i+1}) - g(a_i))\right| \leq \varepsilon.
 \]
 Let us now prove the direct implication. Define $g \in C_0[0, 1]$ as the Henstock-Kurzweil indefinite integral of $f$. Fix $\varepsilon > 0$ and choose a gauge $\delta$ as in the definition of Henstock-Kurzweil integrability. Let $\scrD  = \{([a_i, b_i], x_i) : i \in \{1, \dots, n\}\}$ be a $\delta$-fine division of $[0, 1]$. There are intervals $J_1, \dots, J_m$ such that \[\{[a_i, b_i] : i \in \{1, \dots, n\}\} \cup \{J_k : k \in \{1, \dots, m\}\}\] is partition of $[0, 1]$ into pairwise non-overlapping compact intervals.
 
 For each $k$, there is a gauge $\delta_k$ on $J_k$ such that, for any $\delta_k$-fine tagged partition $\scrP_k$ of $J_k$, we have
 \[
 \left| \scrS( f_{\mid J_k}, \scrP_k) - (\mathrm{HK}) \int_{J_k} f \right| \leq \frac{\varepsilon}{m}.
 \]
 There is no loss of generality in supposing that $\delta_k \leq \delta$ on $J_k$. By the Cousin lemma, such a partition $\scrP_k$ exists; and we fix one for each $k$ for the remainder of the proof. Then
 \[
 \scrP = \scrD \cup \bigcup_{k=1}^m \scrP_k
 \]
 is a $\delta$-fine partition of $[0, 1]$, thus
 \begin{multline*}
 \left|\sum_{i=1}^n  f(x_i) (b_i - a_i) - (g(b_i) - g(a_i)) \right| \\ \leq |\scrS(f, \scrP) - g(1)| + \sum_{k=1}^m \left| \scrS( f_{\mid J_k}, \scrP_k) - (\mathrm{HK}) \int_{J_k} f\right| \leq 2 \varepsilon.
 \end{multline*}
 What we just proved is not exactly the desired conclusion because the absolute values are misplaced. To resolve this, we introduce the subdivisions
 \begin{gather*}
 \scrD^+ = \left\{([a_i, b_i], x_i) : i \in \{1, \dots, n\} \text{ and } f(x_i) (b_i - a_i) - (g(b_i) - g(a_i)) > 0	 \right\}, \\
 \scrD^- = \left\{([a_i, b_i], x_i) : i \in \{1, \dots, n\} \text{ and } f(x_i) (b_i - a_i) - (g(b_i) - g(a_i)) \leq	0 \right\}.
 \end{gather*}
 By what precedes (applied to the divisions $\scrD^+$ and $\scrD^-$ instead of $\scrD$),
 \begin{multline*}
 \sum_{i=1}^n \left| f(x_i) (b_i - a_i) - (g(b_i) - g(a_i)) \right| \leq
 \left| \sum_{([a_i, b_i], x_i) \in \scrD^+}  f(x_i) (b_i - a_i) - (g(b_i) - g(a_i)) \right| \\
 +  \left| \sum_{([a_i, b_i], x_i) \in \scrD^-}  f(x_i) (b_i - a_i) - (g(b_i) - g(a_i)) \right|\leq 4 \varepsilon.
 \end{multline*}
 The conclusion follows.
\end{proof}
We are finally able to present the Pfeffer integral.

\begin{itemize}
\item A \emph{tagged division} of $[0, 1]^d$ (into $BV$-sets) is a finite family  $\scrD = \{(B_i, x_i) : i \in \{1, \dots, n\}\}$, where the $B_i$ are in $\niceBV([0, 1]^d)$, and for each $i$, the tag $x_i$ lies in the essential closure of $B_i$. This means that
\[
\limsup_{r \to 0} \frac{|B_i \cap B(x_i, r)|}{r^d} > 0.
\]
where $B(x_i, r)$ denotes the (closed or open) ball centered at $x_i$ of radius $r$.
\item Let $\eta > 0$. We say that the tagged division $\scrD$ is \emph{$\eta$-regular} if $\reg B_i > \eta$ for each $i$. 
\item In the setting of Pfeffer integration, a \emph{(partial) gauge} is a nonnegative function $\delta \colon [0, 1] \to \R^{\geq 0}$ such that its singular set $\{ \delta = 0 \}$ is $\scrH^{d-1}$ $\sigma$-finite. This novelty is here to ensure nice additivity properties of the Pfeffer integral. We shall not emphasize this point further.
\item The tagged division $\scrD$ is \emph{$\delta$-fine} whenever $\diam B_i < \delta(x_i)$ for each $i$. Because of the strict inequality, it is forbidden to choose tags in the singular set $\{\delta = 0\}$.
\end{itemize}

\begin{defn}
\label{def:pfef}
  A function $f \colon [0, 1]^d \to \R$ is \emph{Pfeffer-integrable} if there exists a charge $\omega$ over $[0, 1]^d$ such that, for all $\eta > 0$, for all $\varepsilon > 0$, there exists a (partial) gauge $\delta$ such that, for all $\eta$-regular $\delta$-fine division $\scrD$ of $[0, 1]^d$, one has
  \[
  \sum_{(B, x) \in \scrD}\left| f(x) |B| - \omega(B) \right| \leq \varepsilon.
  \]
  We say that $\omega$ is the indefinite integral of $f$ and we write $(\mathrm{Pf}) \int_{B} f = \omega(B)$ for any $B \in \niceBV([0, 1]^d)$.
\end{defn}

The formulation is of Saks-Henstock type because, in general, $\delta$-fine partitions do not always exist. For example, when $d=1$ and $\delta(x) = x/2$ on $[0,1]$, it is impossible to construct a $\delta$-fine partition. However, the well-definedness of this definition is based on a powerful Cousin lemma, due to Howard (\cite[Lemma~5]{Howard1990} or \cite[Lemma~2.6.4]{Pfef2}), that guarantees the existence of almost filling $\delta$-fine $\eta$-regular partitions (for $\eta$ smaller than the regularity of cubes), and therefore, the uniqueness of the indefinite integral $\omega$. It is also possible to give a formulation of the Pfeffer integral in terms of Riemann sums, with control over the size of the residual set.

The control of the regularity of the $BV$-sets is there to ensure a divergence theorem. The following version is particularly general: it applies to any $BV$-set---the most general class of domains for which the divergence theorem makes sense---and allows for a singular set of points where the vector field is not differentiable.

\begin{thm}
  Let $v \colon [0, 1]^d \to \R^d$ be a continuous vector field, such that
  \[
	\{x \in [0, 1]^d : v \text{ is not differentiable at } x\}  
  \]
  is of $\scrH^{d-1}$ $\sigma$-finite measure.
   Then $\diver v$ is Pfeffer-integrable and, for all $B \in \niceBV([0, 1]^d)$, one has
  \[
  (\mathrm{Pf}) \int_B \diver v = \int_{\partial^* B} v \cdot n_B \, d \scrH^{d-1}.
  \]
\end{thm}

\begin{proof}
  We define the charge $\omega = \diver v$ (not to be confused with the function also denoted $\diver v$!). Let $\eta > 0$ and $\varepsilon > 0$. By Proposition~\ref{prop:fluxdens}, for each point $x$ at which $v$ is differentiable, there exists $\delta(x) > 0$ such that
  \begin{equation}
  \label{eq:proofdivPf}
  | \diver v(x) |B| - \diver v(B)| \leq \varepsilon |B|
  \end{equation}
  for any $BV$-set $B$ that is $\eta$-regular of diameter less than $\delta(x)$ and such that $x \in B$. (If instead we require that $x$ belongs to the essential closure of $B$, then~\eqref{eq:proofdivPf} still holds by replacing $B$ with $B \cup \{x\}$). If $x$ is not a point of differentiability of $v$, we set $\delta(x) = 0$. Since $\{\delta = 0\}$ is $\scrH^{d-1}$ $\sigma$-finite, $\delta$ is a (partial) gauge. For any $\eta$-regular $\delta$-fine division $\scrD$, one has
  \[
  \sum_{(B, x) \in \scrD} | \diver v(x) |B| - \diver v(B)| \leq \varepsilon \sum_{(B, x) \in \scrD} |B| \leq \varepsilon. \qedhere
  \]
\end{proof}

Because the Pfeffer integral extends the Lebesgue integral, the preceding theorem yields a nontrivial generalized divergence theorem for the Lebesgue integral in the case where $\diver v \in L^1([0,1]^d)$.

Like the Henstock-Kurzweil integral, the space $\mathrm{Pf}([0,1]^d)$ of Pfeffer-integrable functions can be equipped with a natural topology. Its norm, $\|f\|_{\mathrm{Pf}([0,1]^d)}$, is defined as the operator norm of the indefinite integral $\omega$, viewed as an element of the dual space $BV^\infty([0,1]^d)^*$. The following result from~\cite{BuczDePaPfef}, though expected, is quite technical.

\begin{thm}
  The multiplier space for the Pfeffer integral is $BV^\infty([0, 1]^d)$. The continuous linear functionals on $\mathrm{Pf}([0, 1]^d)$ are exactly those maps
  \[
  f \in \mathrm{Pf}([0, 1]^d) \mapsto (\mathrm{Pf}) \int_{[0, 1]^d} fg
  \]
  with $g \in BV^\infty([0, 1]^d)$. This results in a isomorphism $\mathrm{Pf}([0, 1]^d)^* \simeq BV^\infty([0, 1]^d)$.
\end{thm}

In one dimension, the Pfeffer integral is strictly weaker than the Henstock-Kurzweil integral, essentially because compact intervals are replaced by finite unions of compact intervals, \cite[Example~6.9]{Pfef3}. This is a drawback, in view of the well-established status of the Henstock-Kurzweil integral. In dimension \( d \geq 2 \), the Henstock and Pfeffer integrals give rise to different classes of integrable functions, with no inclusion relation between them.

With hardly any additional work, the Pfeffer integral generalizes to a Pfeffer-Stieltjes integral, where the integrator is a charge \(\omega\). This integral was introduced in \cite{Pfef4} under the name \emph{generalized Riemann integral}, motivated by its connection with the multiplier problem. The guiding idea is that a Pfeffer-Stieltjes integral should be approximated by Riemann--Stieltjes sums of the form
\[
(\mathrm{Pf}) \int_{[0,1]^d} f \, \omega \simeq \sum_{(B,x)\in\scrD} f(x)\omega(B),
\]
where \(\scrD\) is an almost filling tagged division into \(BV\)-sets. The precise definition is as follows.

\begin{defn}[Pfeffer-Stieltjes integral]
\label{def:pfefS}
  A function $f \colon [0, 1]^d \to \R$ is \emph{Pfeffer-Stieltjes-integrable} with respect to a charge $\omega$ if there exists a charge $\alpha$ over $[0, 1]^d$ such that, for all $\eta > 0$, for all $\varepsilon > 0$, there exists a (partial) gauge $\delta$ such that, for all $\eta$-regular $\delta$-fine division $\scrD$ of $[0, 1]^d$, one has
  \[
  \sum_{(B, x) \in \scrD}\left| f(x) \omega(B) - \alpha(B) \right| \leq \varepsilon.
  \]
  We say that $\alpha$ is the indefinite integral of $f$ and we write $(\mathrm{Pf}) \int_{B} f  \, \omega = \alpha(B)$ for any $B \in \niceBV([0, 1]^d)$.
\end{defn}

%% file: strongcharges.tex
\section{Analysis in the space of strong charges}
\label{sec:sch}

\subsection{An alternative construction of the space of strong charges}
\label{subsec:strongCh}

Strong charges, \ie{} charges of the form $\diver v$ for $v \in C([0, 1]^d; \R^d)$ can be seen as functionals on $BV([0, 1]^d)$, rather than $BV^\infty([0, 1]^d)$. Indeed, let us define
\begin{equation}
\label{eq:div}
\diver \colon C([0, 1]^d; \R^d) \to BV([0, 1]^d)^*
\end{equation}
by
\[
(\diver v)(u) = - \int_{[0, 1]^d} v \cdot dDu
\]
for $v \in C([0, 1]^d; \R^d)$ and $u \in BV([0, 1]^d)$. We recall that $Du$ is the distributional $\R^d$-valued measure of $u$. In particular, if $v \in C^1([0, 1]^d)$, an integration by part yields
\[
(\diver v)(u) = \int_{[0, 1]^d} u \diver v
\]
\begin{defn}
\label{def:CH01d}
  We define the linear space
  \[
  CH([0, 1]^d) = \operatorname{im} \diver \subset BV([0, 1]^d)^*,
  \]
  with the operator norm inherited from $BV([0, 1]^d)^*$.
\end{defn}
We first show that $CH([0,1]^d)$ is a Banach space, \ie{} that $\diver$ has closed range. By a corollary of the closed range theorem, it suffices to prove that the restricted adjoint
\[
\diver^\# \colon BV([0,1]^d) \to \scrM([0,1]^d;\R^d)
\]
has closed range. Here $\diver^\#$ is defined as the composition of the adjoint operator
\(
\diver^* \colon BV([0,1]^d)^{**} \to \scrM([0,1]^d;\R^d)
\)
with the canonical embedding $BV([0,1]^d) \hookrightarrow BV([0,1]^d)^{**}$. The space $\scrM([0,1]^d;\R^d)$, endowed with the total variation norm, is the dual of $C([0,1]^d)$. For $u \in BV([0,1]^d)$, one has $\diver^\# u = -Du$. Consequently, $\diver^\#$ is an isometry and $CH([0, 1]^d)$ is a Banach space that is isomorphic to the quotient space $C([0, 1]^d) / \ker \diver$. In particular, in dimension $1$, the space $CH([0, 1])$ is isomorphic to $C_0[0, 1]$.

The next result we shall need is the continuous embedding of $L^d([0,1]^d)$ into $CH([0,1]^d)$. This will be useful for the construction of the Faber-Schauder basis in Subsection~\ref{subsec:FB}. This embedding is obtained as follows. Consider the Sobolev-Poincaré embedding
\[
BV([0,1]^d) \to L^{\frac{d}{d-1}}([0,1]^d),
\]
which maps a function $u$ to its restriction to $[0,1]^d$. Its adjoint is the operator
\[
T \colon L^d([0,1]^d) \to BV([0,1]^d)^*,
\]
defined by
\[
T_f(u) = \int_{[0,1]^d} fu \text{ for } f \in L^d([0,1]^d) \text{ and } u \in BV([0,1]^d).
\]
In fact, $T$ maps $L^d([0, 1]^d)$ into $CH([0,1]^d)$:
\begin{equation}
\label{eq:T}
T \colon L^d([0,1]^d) \to CH([0, 1]^d).
\end{equation}
To prove this, it suffices to exhibit a dense subspace $Z \subset L^d([0,1]^d)$ such that $T(Z) \subset CH([0,1]^d)$. We can take $Z = C^\infty([0, 1]^d)$. Indeed, for any $f \in Z$, the divergence equation $\diver v = f$ admits a solution $v \in C^\infty([0, 1]^d; \R^d)$; one can even find a solution of the form $v = (v_1, 0, \dots, 0)$. This implies that $T_f$ is a charge.

The preceding discussion shows non-constructively that the equation $\diver v = f$ admits a continuous solution whenever $f$ is $L^d([0, 1]^d)$. This result is due to Bourgain-Brezis (in a slightly different setting) and this whole discussion is based on some arguments present in~\cite{BourBrez}.

The following theorem shows that $CH([0,1]^d)$ is a predual of $BV([0,1]^d)$. This fact is not surprising: the space $BV([0,1]^d)$ is naturally a dual space, since its unit ball is compact in $L^1([0,1]^d)$. This compactness theorem is a fundamental feature of $BV([0, 1]^d)$, see~\cite[Theorem~5.5]{EvanGari}.

\begin{thm}[Duality]
\label{thm:dualCH}
The canonical map
\[
	\Upsilon \colon BV([0, 1]^d) \to CH([0, 1]^d)^*
\]
that sends $u$ to the map $\omega \mapsto \omega(u)$ is an isometric isomorphism of Banach spaces.
\end{thm} 

\begin{proof}
  We only prove the most difficult part of the argument, namely the surjectivity of $\Upsilon$. Let $\alpha \in CH([0, 1]^d)^*$. Recalling the operator $T$ from~\eqref{eq:T}, the map $\alpha \circ T$ belongs to $L^{d}([0, 1]^d)^*$ and is therefore represented by a map $u \in L^{d/(d-1)}([0, 1]^d)$, satisfying
  \[
  \alpha(T_f) = \int_{[0, 1]^d} fu
  \]
  for all $f \in L^d([0, 1]^d)$. We extend $u$ by zero to $\R^d$ and wish to prove that $u \in BV([0, 1]^d)$. To this end, we consider a vector field $v \in C^1_c(\R^d; \R^d)$
  such that $\|v\|_\infty \leq 1$. Call $g$ the restriction of
  $\diver v$ to $[0, 1]^d$. First, note that
  \[
  T_{g}(\varphi) = \int_{[0, 1]^d} \varphi
  \diver v = \int_{\R^d} \varphi \diver v \leq \|D\varphi\|(\R^d)
  \]
  for all $\varphi \in BV([0, 1]^d)$. This establishes that
  $\|T_{g}\| \leq 1$. Then we
  observe that
  \[
  \int_{[0, 1]^d} u \diver v = \alpha(T_{g}) \leq
  \|\alpha\|.
  \]
  As $v$ is arbitrary, this proves that $\|Du\|(\R^d) \leq
  \|\alpha\| < \infty$ and so $u \in BV([0, 1]^d)$.

  The continuous maps $\alpha$ and $\Upsilon(u)$ coincide on $T(L^d([0,
    1]^d))$, a dense subspace of $CH([0, 1]^d)$. On this
  account, we infer that $\alpha = \Upsilon(u)$. So, $\Upsilon$ is onto.
\end{proof}
 
The space $CH([0,1]^d)$ is a predual of $BV([0,1]^d)$ and is therefore naturally a subspace of $BV([0,1]^d)^*$. Its elements are precisely the linear functionals on $BV([0,1]^d)$ that are weak* continuous. A closer examination shows that, when restricted to the unit ball of $BV([0,1]^d)$, the weak* topology is metrizable by the $L^1$-distance, in accordance with the $BV$ compactness theorem. As a consequence, one obtains the following characterization of strong charges. We will build on this result to define the charges in middle dimension in the final subsection~\ref{subsec:ygt}. Notice that the continuity property in the statement below closely resembles that of Definition~\ref{def:charge}.
\begin{prop}
\label{prop:chargeCONT}
Let $\omega \in BV([0, 1]^d)^*$. The following are equivalent.
\begin{enumerate}
\item $\omega \in CH([0, 1]^d)$.
\item For every sequence $(u_n)$ in $BV([0, 1]^d)$ and $u \in BV([0, 1]^d)$ such that $u_n \to u$ in the $L^1$-sense and $\sup_n \|Du_n\|(\R^d) < \infty$, we have $\omega(u_n) \to \omega(u)$.
\end{enumerate}
\end{prop}
 
\subsection{One-dimensional Haar and Faber-Schauder bases}
\label{subsec:1dHFS}
We have already seen that in dimension $d = 1$, the space $CH([0,1])$ is isomorphic to $C_0[0,1]$. Let us recall some well-known facts about $C_0[0,1]$, in particular that it has a convenient Schauder basis, namely the Faber-Schauder basis. We review this classical one-dimensional result carefully before moving on to the general case in higher dimensions in the next subsection.

The Faber-Schauder system is closely related to the Haar system, which we define first. Let $h_{-1} = 1$ denote the exceptional function of “generation $-1$,” equal to $1$ on $[0,1]$. The only Haar function of generation $0$ is the square-shaped function
\[
h_{0,0} = \ind_{[0,1/2]} - \ind_{[1/2,1]}.
\]
Subsequent Haar functions are obtained by scaling and translating $h_{0,0}$:
\[
h_{n,k}(x) = 2^{n/2}\, h_{0,0}\bigl(2^n x - k\bigr), \quad \text{for } x \in [0,1],\ n \in \mathbb{N},\ k = 0,1,\dots,2^n-1.
\]
Note that $h_{n,k}$ is supported on the dyadic interval $[2^{-n}k, 2^{-n}(k+1)]$.
The factor $2^{n/2}$ ensures that the Haar functions form an orthonormal basis of $L^2([0,1])$. More generally, they constitute a Schauder basis of $L^p([0,1])$ for $1 \le p < \infty$. When $p = 1$, the basis is not unconditional, so the Haar functions must be ordered carefully in lexicographic order: 
\[
h_{-1},\ h_{0,0},\ h_{1,0},\ h_{1,1},\ h_{2,0},\ h_{2,1},\dots
\]
We obtain the Faber-Schauder functions as the indefinite integrals of the Haar functions:
\begin{equation}
\label{eq:faberfnk}
f_{-1}(x) = \int_0^x h_{-1}, \qquad f_{n,k}(x) = \int_0^x h_{n,k},
\end{equation}
for $x \in [0, 1]$, $n \in \N$, and $k \in \{0, \dots, 2^n-1\}$. It is important to note that the Faber--Schauder functions are localized in space: each one has the same support as the corresponding Haar function. A consequence of this localization is the following estimate for linear combinations of Faber--Schauder functions of generation $n$:
\begin{equation}
\label{eq:fnkmax}
\left\| \sum_{k=0}^{2^n - 1} a_{n,k} f_{n,k} \right\|_\infty = 2^{- \frac{n}{2} - 1} \max \left\{ |a_{n,k}| : 0 \leq k \leq 2^n - 1 \right\}.
\end{equation}
It was one of the first nontrivial examples of a Schauder basis, as established by Schauder in 1927.
\begin{thm}[Schauder, \cite{Schau}]
The family of Faber-Schauder functions, ordered lexicographically, is a Schauder basis of $C_0[0, 1]$, \ie{} each $f \in C_0[0, 1]$ can be uniquely decomposed as a series
\[
f = a_{-1} f_{-1} + \sum_{n, k} a_{n,k} f_{n,k}
\]
that is convergent in $C_0[0, 1]$.
\end{thm}
Below we give three classical applications of these bases. We also refer to the paper~\cite{GubiImkePerk}, which contains a discussion of many further applications.
\begin{enumerate}
\item \emph{Lévy-Ciesielski construction of the Brownian motion}: the Brownian motion can be constructed as a random expansion in the Faber–Schauder basis, using independent standard Gaussian coefficients:
\[
B_t = X_{-1} f_{-1} + \sum_{n = 0}^\infty \sum_{k=0}^{2^n - 1} X_{n,k} f_{n,k}.
\]
The resulting process $B$ has continuous paths, independent increments, and the correct Gaussian covariance structure. This construction makes explicit the multiscale nature of Brownian motion, with each generation adding finer random oscillations.
\item \emph{Characterization of Hölder function spaces}: a function $f \in C_0[0, 1]$ is $\gamma$-Hölder continuous (with $0 < \gamma < 1$) if and only if its Faber-Schauder coefficients decay at the rate $2^{-n(\gamma - 1/2)}$ uniformly in the spatial index $k$. More precisely, we have the following equivalence of norms:
\[
[f]_{C^\gamma} \simeq \max \left\{ |a_{-1}|, \sup_{n,k} 2^{n(\gamma -1/2)} |a_{n,k}| \right\},
\]
where
\[
[f]_{C^\gamma} = \sup \left\{ \frac{|f(y) - f(x)|}{|y - x|^\gamma} : x , y \in [0, 1] \text{ and } x \neq y \right\}.
\]
In particular, there is an isomorphism between $C^\gamma[0, 1]$ and the sequence space $\ell^\infty$, known as the Ciesielski isomorphism, see~\cite{Cies}.
\item \emph{Young integral}: let $\beta, \gamma \in (0, 1)$ be exponents such that $\beta + \gamma > 1$. Consider a function $f \in C^\beta[0, 1]$, with Haar coefficients $a_{-1}, a_{n,k}$ and another function $g \in C^\gamma[0, 1]$, vanishing at $0$, with Faber-Schauder coefficients $b_{-1}$, $b_{n,k}$. We define the \emph{Young integral}
\[
\int_0^1 f \, dg = a_{-1} b_{-1} + \sum_{n=0}^\infty \sum_{k=0}^{2^n - 1} a_{n,k} b_{n,k},
\]
where the combined regularity of $f$ and $g$ guarantees the convergence of the above series.
\end{enumerate}
The almost sure $\gamma$-Hölder regularity of the Brownian motion, for $\gamma < 1/2$ can be deduced from (2). It should be noted that, in application (3), the derivative-to-indefinite-integral relationship between Haar and Faber-Schauder functions is essential. The existence of a multidimensional Faber–Schauder basis, which plays in $CH([0,1]^d)$ a role analogous to the original Faber–Schauder basis in $C_0[0,1]$, opens the way to investigating (1), (2) and (3) in higher dimensions, as we shall do in the next subsections.

\subsection{A multidimensional Faber-Schauder basis} 
\label{subsec:FB}

We begin by recalling the $d$-dimensional (isotropic) Haar system on $[0,1]^d$.
It consists of the exceptional constant function $h_{-1} = \ind_{[0,1]^d}$.
At generation $0$, there are $2^d-1$ Haar functions, indexed by the set $E = \{0,1\}^d \setminus \{(0,\dots,0)\}$.
These functions are precisely all tensor products of one-dimensional Haar
functions of generations $-1$ and $0$, excluding the constant function.
To make this explicit, we introduce the one-dimensional functions
\[
\psi_0 = \ind_{[0,1/2]} - \ind_{[1/2,1]},
\qquad
\psi_1 = \ind_{[0,1]}.
\]
For $e = (e_1,\dots,e_d) \in E$ and $(x_1, \dots, x_d) \in [0, 1]^d$, we then define
\[
h_{0,0,e}(x_1,\dots,x_d) = \prod_{i=1}^d \psi_{e_i}(x_i).
\]
As an illustration, when $d=2$ the three generation-$0$ Haar functions are
\begin{gather*}
h_{0,0,(0,1)}= \ind_{[0,1/2]\times[0,1]} - \ind_{[1/2,1]\times[0,1]},\\
h_{0,0,(1,0)}= \ind_{[0,1]\times[0,1/2]} - \ind_{[0,1]\times[1/2,1]},\\
h_{0,0,(1,1)}= \ind_{[0,1/2]^2} + \ind_{[1/2,1]^2} - \ind_{[0,1/2]\times[1/2,1]} - \ind_{[1/2,1]\times[0,1/2]}.
\end{gather*}
These functions capture, respectively, vertical, horizontal, and diagonal
oscillations of a function on $[0,1]^2$.

For any integer $n \geq 0$, let $\{Q_{n,k} : k = 0, \dots, 2^{nd} - 1\}$ denote the collection of (closed) dyadic cubes in $[0, 1]^d$ of generation $n$, that is, cubes of sidelength $2^{-n}$. We denote by $y_{n,k}$ the lower-left corner of $Q_{n,k}$. The Haar functions of generation $n$ are defined by
\begin{equation}
\label{eq:hnke}
h_{n,k,e}(x) = 2^{nd/2} h_{0, 0, e}\left(2^n ( x - y_{n,k} ) \right) 
\end{equation}
for $x \in [0, 1]^d$, $k \in \{0, \dots, 2^{nd}-1\}$ and $e \in E$. To make sense of~\eqref{eq:hnke}, we first extend the function $h_{0,0,e}$ from $[0,1]^d$ to the whole space $\R^d$ by setting it equal to zero outside $[0,1]^d$. We remark that the support of $h_{n,k,e}$ is the dyadic cube $Q_{n,k}$.

It can be shown that the Haar system forms a Hilbert basis of $L^2([0,1]^d)$. More generally, it is a Schauder basis of
$L^p([0,1]^d)$ for every $1 \leq p < \infty$. However, when $p=1$ the basis is not unconditional, as in dimension $d=1$.
Consequently, in this case the Haar functions must be ordered
with care, by increasing generation number.

Recall the map $T \colon L^d([0, 1]^d) \to CH([0, 1]^d)$ from Subsection~\ref{subsec:strongCh}. Haar functions are bounded, therefore in $L^d([0, 1]^d)$.
\begin{thm}[{\cite[Theorem~5.1]{BouaDePa}}]
\label{thm:faberCH}
  The family of strong charges $T_{h_{-1}}$ and $T_{h_{n,k,e}}$, for all $n \in \N$, $k \in \{0, \dots, 2^{nd}-1\}$ and $e \in E$, ordered by increasing generation number $n$, is a Schauder basis of $CH([0, 1]^d)$. Each strong charge $\omega$ is uniquely decomposed
  \[
  \omega = \omega(h_{-1}) T_{h_{-1}} + \sum_{n=0}^\infty \sum_{k=0}^{2^{nd}-1} \sum_{e \in E} \omega(h_{n,k,e}) T_{h_{n,k,e}}.
  \]
\end{thm}
The basis introduced in Theorem~\ref{thm:faberCH} is called the \emph{Faber-Schauder} basis of $CH([0,1]^d)$. We retain the intuition that the charges $T_{h_{n,k,e}}$ may be viewed, in a sense, as indefinite integrals of the corresponding Haar functions $h_{n,k,e}$; see Example~\ref{ex:chargeL1}.

When $d=1$, recall that $CH([0,1])$ is isomorphic to $C_0[0,1]$. Under this identification, one readily checks that the Faber-Schauder basis coincides with the usual basis of functions defined in~\eqref{eq:faberfnk}.

When viewed as set functions, the strong charges $T_{h_{n,k,e}}$ are well-localized: their support coincides with the dyadic cube $Q_{n,k}$, just as for the corresponding Haar functions. Note that this nice localization property would be lost if, instead of $T_{h_{n,k,e}}$, we considered the indefinite integrals of $h_{n,k,e}$ in the sense of~\eqref{eq:primitiveRect}. One manifestation of this property is the following estimate that we can prove for linear combinations of strong charges $T_{h_{n,k,e}}$ sharing the same generation $n$:
\begin{equation}
\label{eq:fnkmaxDimD}
\left\| \sum_{k,e} a_{k,e} T_{h_{n,k,e}} \right\| \leq C 2^{n\left(\frac{d}{2}-1\right)} \max_{k,e} \left\{ |a_{k,e}| : 0 \leq k \leq 2^{nd} - 1 \text{ and } e \in E \right\},
\end{equation}
see \cite[Proposition~6.2]{BouaDePa}. This generalizes~\eqref{eq:fnkmax}. The multiplicative constant in~\eqref{eq:fnkmaxDimD} grows exponentially with the dimension $d$, providing a weak indication that the situation may deteriorate in high dimensions.

Once we have a Faber-Schauder basis in $CH([0,1]^d)$, the next step is to define a notion of ``Hölder charge'' (in contrast with the terminology of \cite{Bouafia2025Young}, these objects will henceforth be called fractional charges). Inspired by Ciesielski's isomorphism, these should be the charges whose Faber-Schauder coefficients decay exponentially with the generation $n$. However, we will take a slightly indirect route to define such charges, first examining the situation in dimension one.

Every charge on $[0, 1]$ has the form $dv$, for some $v \in C_0[0, 1]$ by Example~\ref{ex:chargedim1}. The following discussion characterizes the fact that $v$ is $\gamma$-Hölder continuous, for $\gamma \in (0, 1)$, in term of the charge $dv$. For any $BV$-set $B = \bigcup_{i=1}^n [a_i, b_i]$, one has
\[
|dv(B)| \leq \sum_{i=1}^n |dv([a_i, b_i])| \leq C \sum_{i=1}^n |a_i - b_i|^\gamma \leq C n^{1 - \gamma} |B|^\gamma
\]
by Hölder inequality. The number $n$ of intervals is half the perimeter of $B$, therefore one has the interpolation inequality
\begin{equation}
\label{eq:dvB}
|dv(B)| \leq C \|B\|^{1 - \gamma} |B|^\gamma.
\end{equation}
Conversely, if $v$ satisfies~\eqref{eq:dvB} for each $B \in \niceBV([0, 1])$, then it is clearly $\gamma$-Hölder continuous. This motivates the following definition.

\begin{defn}
\label{def:chgamma}
A \emph{$\gamma$-fractional charge} is a linear map $\omega \colon BV([0, 1]^d) \to \R$ for which there exists a constant $C \geq 0$ such that
\[
\forall u \in BV([0, 1]^d), \qquad |\omega(u)| \leq C \| Du\|(\R^d)^{1-\gamma} \|u\|_1^\gamma.
\]
The best such constant is denoted $\|\omega\|_{CH^\gamma}$. The space of $\gamma$-fractional charges is denoted $CH^\gamma([0, 1]^d)$.
\end{defn}

First, we note that a $\gamma$-fractional charge is indeed a strong charge, because the continuity property stated in Proposition~\ref{prop:chargeCONT} is satisfied.
We now derive some consequences from this definition. We shall consider a $\gamma$-fractional charge both as a set function or a functional on $BV([0, 1]^d)$. Thus we may write $\omega(B)$ for $\omega(\ind_B)$. Let us estimate the size of $\omega(K)$, whenever $K$ is a dyadic cube of $[0, 1]^d$:
\[
|\omega(K)| \leq \|\omega\|_{CH^\gamma} \|K\|^{1- \gamma} |K|^\gamma = C \|\omega\|_{CH^\gamma} |K|^\delta
\]
where
\[
\delta = \frac{d-1+\gamma}{d}.
\]
More generally, let $B$ be any $BV$-subset of $[0, 1]^d$. In this case, $\gamma$-fractionality tells us that
\[
|\omega(B)| \leq \|\omega\|_{CH^\gamma} \|B\|^{1 - \gamma} |B|^\gamma \leq \frac{\| \omega \|_{CH^\gamma}}{ (\operatorname{isop}B)^{1 - \gamma}} |B|^\delta,
\]
where $\operatorname{isop} B$ is the \emph{isoperimetric coefficient} of $B$:
\begin{equation}
\label{eq:isop}
\operatorname{isop} B = \frac{|B|^{(d-1)/d}}{\| B \|}.
\end{equation}
In a nutshell, $|\omega(B)|$ is controlled by the power $|B|^\delta$, up to a multiplicative ``constant'' that depends on the isoperimetric regularity of $B$.

The following lemma will be key in the next subsections. It allows to fabricate fractional charges very easily, by extending set functions defined solely on dyadic cubes.

\begin{lem}
\label{lem:37}
 Let $\gamma \in \left( 0, 1 \right)$, $\delta = (d-1 + \gamma)/d$ and
     $\omega \colon \{\text{dyadic cubes}\subset [0, 1]^d\} \to \R$ a set function that satisfies the following properties:
    \begin{enumerate}
        \item \emph{Additivity}: for any dyadic cube $K \subset [0, 1]^d$,
        \[
    \omega(K) = \sum_{L \text{ children of } K} \omega(L).
        \]
        \item \emph{Hölder control}: there is a constant $C \geq 0$ such that 
        \[
        |\omega(K)| \leq C |K|^\delta
        \]
        for all dyadic cubes $K \subset [0, 1]^d$.
    \end{enumerate}
    Then $\omega$ has a unique extension to $\niceBV([0, 1]^d)$ (or $BV([0, 1]^d)$) that is a $\gamma$-fractional charge.
\end{lem}

A finite union of dyadic cubes is called a \emph{dyadic figure}. By additivity, any set function
$\omega$ satisfying the hypotheses of Lemma~\ref{lem:37} extends canonically to the class of dyadic figures.
A theorem of De Giorgi provides the following approximation result: every $BV$-set
$B \subset [0,1]^d$ admits a sequence $(F_n)$ of dyadic figures such that
\[
|B \ominus F_n| \to 0 \quad\text{and}\quad \sup_n \|F_n\| \le C \|B\|.
\]
The natural question, then, is whether the Hölder control in~(2) is sufficient to extend
$\omega$ by density to $\niceBV([0,1]^d)$.

\begin{proof}[Idea of proof]
   We write the formal Faber-Schauder decomposition
   \begin{equation}
   \label{eq:fFSd}
   \omega =  \omega(h_{-1}) T_{h_{-1}} + \sum_{n=0}^\infty \sum_{k=0}^{2^{nd}-1} \sum_{e \in E} \omega(h_{n,k,e}) T_{h_{n,k,e}}
   \end{equation}
   as in Theorem~\ref{thm:faberCH}. It is possible to make sense of the coefficients $\omega(h_{-1})$ and $\omega(h_{n,k,e})$, because the Haar functions are piecewise constant on dyadic cubes. Moreover, the Hölder control yields estimate of the coefficients $|\omega(h_{n,k,e})|$. Using~\eqref{eq:fnkmaxDimD}, we then show that the series in~\eqref{eq:fFSd} converges in $CH([0, 1]^d)$. Consequently, the left-hand side of~\eqref{eq:fFSd} is well-defined as an element of $CH([0, 1]^d)$ and provides the desired extension. Proving that $\omega$ is actually $\gamma$-fractional requires an interpolation argument that is to be found in~\cite[Theorem~3]{Bouafia2025Young}.
\end{proof}

It is also possible to prove that $CH^\gamma([0, 1]^d)$ consists exactly of those strong charges $\omega$ whose Faber-Schauder coefficients have the following decay:
\[
|\omega(h_{n,k,e})| = O \left( 2^{n\left( 1 - \gamma - \frac{d}{2} \right)}\right).
\]
Therefore, the space $CH^\gamma([0,1]^d)$ is isomorphic to $\ell^\infty$,
in the same way that $C^\gamma([0,1])$ is related to $\ell^\infty$ via the Ciesielski
isomorphism. In particular, $CH^\gamma([0,1]^d)$ is a dual Banach space (indeed, a bidual, see~\cite[Proposition~7.2]{Boua2}).

Moreover, just as the unit ball of $C^\gamma([0,1])$ is compact in $C([0,1])$ by the
Arzelà-Ascoli theorem (which reflects the weak* topology), the same holds for the
unit ball of $CH^\gamma([0,1]^d)$ in $CH([0,1]^d)$. We now formalize this idea.
The following notion of weak* convergence will be useful in stating the continuity
properties of the Young integral.
\begin{defn}
\label{def:weak*}
A sequence $(\omega_n)$ in $CH^\gamma([0,1]^d)$ is said to converge \emph{weakly*} to
$\omega \in CH^\gamma([0,1]^d)$ if $(\omega_n)$ is bounded in
$CH^\gamma([0,1]^d)$ and $\omega_n \to \omega$ in $CH([0,1]^d)$.
\end{defn}
Finally, we briefly discuss the characterization of $\gamma$-fractional charges. 
It was shown in~\cite[Theorem~11]{Bouafia2025Young} that they are precisely 
the divergences of $\gamma$-Hölder continuous vector fields. More specifically, 
the operator $\diver$ defined in~\eqref{eq:div} induces a surjective map
\[
C^\gamma([0,1]^d;\mathbb{R}^d) \longrightarrow CH^\gamma([0,1]^d).
\]
Consequently, $\gamma$-fractional charges can be interpreted as distributions of negative Hölder exponent $\gamma-1$. Nevertheless, we retain the term \emph{$\gamma$-fractional charges} for historical reasons related to the Pfeffer integral and because, as we shall see in Section~\ref{sec:git}, these objects generalize naturally to any codimension, where they appear to have no analogue in any standard function space.

\subsection{Increments of stochastic processes over $BV$-sets}
\label{subsec:incr} Let $f \colon [0,1]^d \to \mathbb{R}$ be any function. When studying multipliers for the Henstock integral, 
we introduced the notion of the \emph{rectangular increment} of $f$ over a rectangle 
$I = \prod_{i=1}^d [a_i, b_i]$:
\[
\Delta_f(I) = \sum_{(c_i) \in \prod_{i=1}^d \{a_i, b_i\}} 
(-1)^{\delta_{a_1, c_1}} \cdots (-1)^{\delta_{a_d, c_d}} f(c_1, \dots, c_d).
\]  
The set function $\Delta_f$ extends by additivity to a broader class of sets, namely \emph{rectangular figures}, 
\ie{}, finite unions of rectangles. In many situations, one wishes to integrate a function with respect to the 
increments of $f$, as measured by $\Delta_f$. For integrals constructed using partitions of rectangles, 
this object is sufficient. However, one may naturally ask whether $\Delta_f$ can be turned into objects better suited to other integration theories:
\begin{itemize}
\item Can $\Delta_f$ be extended to a Borel measure on $[0,1]^d$?
\item Can $\Delta_f$ be extended to a charge defined on $\niceBV([0,1]^d)$?
\end{itemize}
The first question is generally too strong—particularly if $f$ is a sample path of a typical multiparameter random process. 
We therefore focus on the second question. This is very natural, since charges are the  integrators 
for the Pfeffer-Stieltjes integral.

\begin{defn}
A function $f \colon [0, 1]^d \to \R$ is said to be \emph{chargeable} (resp. \emph{$\gamma$-fractionally chargeable}) if $\Delta_f$ extends to a charge (resp. a $\gamma$-fractional charge).
\end{defn}

In fact, Lemma~\ref{lem:37} may already provide an answer to the second question. It is no coincidence that the proof of Lemma~\ref{lem:37} relies on the Faber-Schauder basis of $CH([0, 1]^d)$: in general, a Schauder basis gives a criterion for membership in a Banach space. By combining Lemma~\ref{lem:37} with a standard Borel-Cantelli argument, one obtains the following Kolmogorov-type chargeability theorem.

\begin{thm}[{\cite[Theorem~9.1]{BouaDePa}}]
\label{thm:kolmo}
Let $X$ be a stochastic process indexed on $[0, 1]^d$, with continuous sample paths. Let $q > 0, C \geq 0, \eta > 0$ such that
        \[
        \frac{d-1}{d} < \frac{\eta}{q} \leq 1 \qquad \text{and} \qquad
        \mathbb{E}\left(|\Delta_X K|^q\right) \leq C |K|^{1 + \eta}
        \]
        for dyadic cubes $K$. Then a.s, $X$ is $\gamma$-fractionally chargeable for any $0 < \gamma < d\eta/q - (d-1)$.
\end{thm}
In particular, this theorem applies to the 
\emph{fractional Brownian sheet} $B^H = \{B^H_t : t \in [0,1]^d\}$, 
a centered Gaussian process indexed by $[0,1]^d$, whose regularity is 
determined by the Hurst exponents $H = (H_1, \dots, H_d) \in (0,1)^d$: 
larger values of $H_i$ correspond to smoother behavior in the $i$-th direction. 
The covariance of $B^H$ is given by
\[
\operatorname{Cov}\left(B^H_s, B^H_t\right) = \prod_{i=1}^d \frac12 \left( s_i^{2H_i} + t_i^{2H_i} - |t_i - s_i|^{2H_i} \right)
\]
for $s, t \in [0, 1]^d$. The fractional Brownian sheet generalizes the classical Brownian sheet, 
which corresponds to the case $H = (1/2, \dots, 1/2)$.

The covariance function is designed so that the variance of the rectangular increments 
over any rectangle $I = \prod_{i=1}^d [a_i, b_i]$ has the simple product expression
\[
\mathrm{Var}\big(\Delta_{B^H}(I)\big) = \prod_{i=1}^d |b_i - a_i|^{2H_i}.
\]
One observes that, on average, $\Delta_{B^H}$ behaves like a fractional charge:
$\Delta_{B^H}(I)$ is comparable to $|I|^{\bar H}$ whenever $I$ is a cube, or close
to a cube, where $\bar H$ denotes the mean of the Hurst coefficients. This estimate,
however, deteriorates when $I$ is a thin rectangle. In fact, Theorem~\ref{thm:kolmo}
can be used to establish the positive part of the following result.
\begin{thm}[{\cite[Theorems~10.2 and~10.5]{BouaDePa}}] Recall $\bar{H} = \frac{H_1 + \cdots + H_d}{d}$.
\begin{enumerate}
\item If $\bar{H} > \frac{d-1}{d}$, then a.s. the sample paths  of $B^H$ are $\gamma$-fractionally chargeable for all $\gamma \in (0, H_1 + \cdots + H_d - (d-1))$.
\item If $\bar{H} \leq \frac{d-1}{d}$, then a.s. the sample paths of $B^H$ are not chargeable.
\end{enumerate}
\end{thm}
Altogether, chargeability requires a high degree of regularity, and this requirement becomes increasingly restrictive as the dimension grows. In particular, in dimension $d \geq 2$, the Brownian sheet is not chargeable. This does not contradict the fact that its increments over arbitrary measurable sets can still be defined as a \emph{random measure}; almost sure chargeability is a pathwise property of random processes. This negative result reflects the impossibility of extending the Lévy-Ciesielski construction of Brownian motion to the Brownian sheet within the space $CH([0,1]^d)$.

Also observe that, when viewed as continuous functions, the sample paths of $B^H$ are almost surely
$\gamma$-Hölder continuous for every
$\gamma < \min\{H_1,\dots,H_d\}$. In this setting, the regularity is governed
by the worst Hurst coefficient, in contrast with the previous result. 

When $\bar H > \frac{d-1}{d}$, the fact that $\Delta_{B^H}$ is not only almost surely chargeable, but even $\gamma$-fractionally so, implies that one can define its increments over sets that are more irregular than merely $BV$-sets. As far as the use of the Pfeffer--Stieltjes integral with respect to $\Delta_{B^H}$ is concerned, this question is mostly a matter of curiosity. It will be clear in Example~\ref{ex:fractSobolev} that the natural class of sets to consider consists of those with finite fractional perimeter, \ie{}, sets whose indicator functions belong to an appropriate fractional Sobolev space.

\begin{quest}
  When $\bar{H} > \frac{d-1}{d}$, what is the exact regularity of $\Delta_{B^H}$ in term of the decay of its Faber-Schauder coefficients?
\end{quest}

\subsection{Young integral with respect to a fractional charge}
As mentioned in point~(3) of Subsection~\ref{subsec:1dHFS}, it is entirely possible to define a Young
integral $(\mathrm{Y})\int_\bullet f\,\omega$ by decomposing a Hölder function $f$ in the Haar
basis and a fractional charge in the Faber-Schauder basis. This leads to the following result.

\begin{thm}
\label{thm:young}
Let $\beta, \gamma \in (0, 1)$ such that $\beta + \gamma > 1$.
        There is a unique map
        \[
        (\mathrm{Y}) \int_{\bullet} \colon C^\beta([0, 1]^d) \times CH^\gamma([0, 1]^d) \to CH^\gamma([0, 1]^d)
        \]
        such that
        \begin{enumerate}
            \item $(\mathrm{Y}) \int_\bullet f \, T_g = T_{fg}$ for $f \in C^\beta([0, 1]^d)$ and $g \in L^d([0, 1]^d)$;
            \item \emph{Weak* continuity}: if $(f_n)$ converges weakly* to $f$ in $C^\beta([0, 1]^d)$ and $(\omega_n)$ converges weakly* to $\omega$ then $(\mathrm{Y}) \int_\bullet f_n \, \omega_n \to (\mathrm{Y}) \int_\bullet f \, \omega$ weakly*.
        \end{enumerate}
\end{thm}

Weak* convergence in $C^\beta([0,1]^d)$ means that $f_n \to f$ uniformly on
$[0,1]^d$ and that the sequence $(f_n)$ is bounded in $C^\gamma([0,1]^d)$. The notion
of weak* convergence in $CH^\gamma([0,1]^d)$ was introduced in
Definition~\ref{def:weak*}.

We shall provide an alternative proof of Theorem~\ref{thm:young} using a sewing lemma, which
aligns more closely with the modern approach in rough path theory. In general, a
sewing lemma is a tool that allows one to turn an almost additive object into a truly
additive one. While its statement is often simple, its consequences are nonetheless
powerful. In the context of charges, we can formulate the following version of the
sewing lemma.

\begin{lem}
\label{lem:sewing}
Let $\eta \colon \{\text{dyadic cubes} \subset [0, 1]^d\} \to \R$ be a function. Suppose $\varepsilon > 0$, $C \geq 0$ such that
    \begin{enumerate}
        \item \emph{Almost finite additivity}: for every dyadic cube $K$,
        \[
        \left| \eta(K) - \sum_{L \text{ children of }K} \eta(L) \right| \leq C |K|^{1 + \varepsilon}
        \]
        \item \emph{Hölder control}: for every dyadic cube $K$, one has
        \[
        |\eta(K)| \leq C |K|^\delta
        \]
    \end{enumerate}
    Then there is a unique $\omega \in CH^\gamma([0, 1]^d)$ such that $|\omega(K) - \eta(K)| \leq \kappa C |K|^{1 + \varepsilon}$ (for some $\kappa = \kappa(d)$).
\end{lem}

\begin{proof}[Sketch of proof]
 For each dyadic cube $K$, we define
 \[
 \omega(K) = \lim_{n \to \infty} \sum_{\substack{L \text{ grandchildren of } K \\ \text{ of generation }n}} \eta(L)
 \]
 One can verify that the limit on the right-hand side exists as the limit of a Cauchy
sequence, thanks to~(2). It is then possible to apply Lemma~\ref{lem:37} to extend $\omega$ to a $\gamma$-fractional charge. Note however that the proof of Lemma~\ref{lem:sewing} ultimately relies on the properties of the Faber-Schauder basis of $CH([0, 1]^d)$.
\end{proof}

Now, the Young integral in Theorem~\ref{thm:young} can be constructed as follows: pick for any dyadic cube $K$ of $[0, 1]^d$ a point $x_K \in K$, and define the set function $\eta \colon K \mapsto f(x_K) \omega(K)$. Set $\delta = \frac{d-1 + \gamma}{d}$ the Hölder exponent associated to the fractional exponent $\gamma$. Then the set function $\eta$ is almost additive, as
\begin{align*}
\left|f(x_K) \omega(K) - \sum_{L \text{ children of }K} f(x_L) \omega(L)\right| & \leq \sum_{L \text{ children of }K} |f(x_K) - f(x_L)| \, |\omega(L)| \\
& \leq C (\diam K)^\beta |K|^\delta \\
& \leq C |K|^{\frac{d-1 + \beta + \gamma}{d}}
\end{align*}
Lemma~\ref{lem:sewing} also provides a Young-Loeve-type estimate of the error
\[
\left| (\mathrm{Y}) \int_K f \, \omega - f(x_K) \omega(K) \right| \leq C |K|^\delta.
\]
It is natural to ask whether such an estimate holds over sets more general than dyadic cubes. With a bit of extra work, one can prove the following result:
\begin{equation}
\label{eq:youngloeve}
\left| (\mathrm{Y}) \int_B f \,\omega - f(x) \,\omega(B) \right|
\le \frac{C [f]_{C^\beta} \|\omega\|_{CH^\gamma}}{(\operatorname{isop} B)^{1 - \gamma}}
|B|^\delta (\operatorname{diam} B)^\beta,
\end{equation}
for any $BV$-set $B$ and any $x$ in the essential closure of $B$. For the error in~\eqref{eq:youngloeve} to be of order strictly greater than $1$ in $|B|$, the regularity 
\[
\reg B = \frac{|B|}{\|B\| \diam B}
\]
must not be too small. Recall that this quantity was introduced in the definition of the Pfeffer integral to ensure a divergence theorem. Interestingly, it also appears in the seemingly unrelated context of Young integration. In fact, using~\eqref{eq:youngloeve}, one easily proves:
\begin{prop}
Under the hypotheses of Theorem~\ref{thm:young}, the function $f$ is Pfeffer-Stieltjes integrable with respect to $\omega$, and the Pfeffer-Stieltjes and Young indefinite integrals coincide.
\end{prop}
Note that $\reg B$ provides a stronger control than $\operatorname{isop} B$ defined in~\ref{eq:isop}: for example, the union of two distant balls may have a large isoperimetric coefficient but low $\reg$ coefficient.

%% file: geomint.tex
\section{Geometric integration theories}
\label{sec:git}

A geometric integration theory views integration as a duality between geometric domains and algebraic data. On one side are chains, representing generalized oriented objects such as manifolds, possibly with singularities; on the other are cochains, which assign numerical values to these objects. Integration is precisely the evaluation of a cochain on a chain. The boundary operator on chains and the exterior differentiation (or coboundary) operator on cochains are related by a Stokes-type identity. Different theories arise by choosing which chains and cochains are allowed, with the aim of extending integration beyond smooth settings. In this section, we present three of them:
\begin{itemize}
  \item Whitney's theory of flat chains and flat cochains, which in some sense generalizes the Lebesgue integral to higher codimension.
  \item Normal chains and charges in middle dimension, which take an essentially opposite viewpoint; they were meant to provide a functional-analytic setting for non-absolute integration in arbitrary codimension.
  \item Fractional currents and fractional charges, which interpolate between the two preceding theories and generalize Young integration to any codimension.
\end{itemize}

\subsection{Whitney's theory of flat chains and flat cochains}
\label{subsec:whitney}
All the results presented here can be found in~\cite{Whit}.
We begin by defining polyhedral chains, and then construct the space of flat chains as their completion with respect to a suitable norm.

A non-empty subset $\sigma \subset [0, 1]^d$ is called a \emph{cell} if it can be written as a finite intersection of closed half planes in $\R^d$. Its \emph{affine hull}, $\langle \sigma \rangle$, is the minimal affine subspace of $\R^d$ containing $\sigma$, and the dimension of $\sigma$ is that of its affine hull. We say that $\sigma$ is an $m$-cell if it has dimension $m$. An \emph{oriented $m$-cell} is an $m$-cell $\sigma$, together with a choice of orientation of $\langle \sigma \rangle$ (if $m=0$, orienting $\sigma$ simply amounts to choosing $+1$ or $-1$). We typically denote an oriented cell $\llbracket \sigma \rrbracket$, leaving the orientation implicit in the notation. For $0$-cells, that are just singletons, we write $\llbracket x \rrbracket$ instead of $\llbracket \{ x \} \rrbracket$ for the cell $\{x\}$ with the positive orientation.

The space of \emph{polyhedral $m$-chains} in $[0, 1]^d$ is the real linear space spanned by oriented $m$-cells, in the sense that each polyhedral $m$-chain admits a decomposition
\begin{equation}
\label{eq:polyhedralChain}
T = \sum_{i \in I} a_i \, \llbracket \sigma_i \rrbracket,
\end{equation}
where $I$ is finite, $a_i \in \mathbb{R}$ and $\llbracket \sigma_i \rrbracket$ are oriented $m$-cells, modulo the following identifications:
\begin{enumerate}
    \item $\llbracket \sigma \rrbracket = -\, \llbracket \sigma' \rrbracket$ if $\llbracket \sigma \rrbracket$ and $\llbracket \sigma' \rrbracket$ are the same cell with opposite orientations.
    \item $\llbracket \sigma \rrbracket = \llbracket \sigma_1
      \rrbracket + \llbracket \sigma_2 \rrbracket$ if $\llbracket
      \sigma \rrbracket$ is obtained by gluing $\llbracket \sigma_1
      \rrbracket$ and $\llbracket \sigma_2 \rrbracket$. This means that
      $\sigma = \sigma_1 \cup \sigma_2$ (this alone guarantees that
      $\langle \sigma \rangle = \langle \sigma_1 \rangle = \langle
      \sigma_2 \rangle$), where $\sigma_1 \cap \sigma_2$ is a common
      face of $\sigma_1$ and $\sigma_2$ and the orientations of
      $\llbracket\sigma\rrbracket$, $\llbracket\sigma_1\rrbracket$
      and $\llbracket\sigma_2\rrbracket$ are identical.
\end{enumerate}
The space of polyhedral \(m\)-chains is denoted by
\(\bP_m([0,1]^d)\). We adopt the convention that
\(\bP_m([0,1]^d)=0\) if \(m \notin \{0,1,\dots,d\}\).
Each polyhedral \(m\)-chain admits a representation of the
form~\eqref{eq:polyhedralChain}, where for any distinct \(i,j\),
either \(\sigma_i \cap \sigma_j = \emptyset\) or
\(\sigma_i \cap \sigma_j\) is a common face of \(\sigma_i\) and
\(\sigma_j\), of dimension between \(0\) and \(m-1\). The \emph{mass} of $T$ is then
\[
\bM(T) = \sum_{i \in I} |a_i| \scrH^{m}(\sigma_i).
\]
It is not hard to check that $\bM$ is a norm on $\bP_m([0, 1]^d)$.

If \(m \in \{1, \dots, d\}\), each \(m\)-cell has finitely many
\((m-1)\)-dimensional faces. An orientation of the cell induces
orientations on its faces, which gives rise to a boundary operator
\[
\partial \colon \bP_{m}([0, 1]^d) \to \bP_{m-1}([0, 1]^d).
\]
We also set \(\partial = 0\) if \(m \notin \{1,\dots,d\}\).
One checks that, in general, \(\partial \circ \partial = 0\).
However, the boundary operator is not \(\bM\)-continuous.
We therefore introduce the largest norms \(\bF\) on each space
\(\bP_m([0,1]^d)\) that satisfy the following axioms:
\begin{enumerate}
\item \emph{$\bF$-convergence is weaker than mass-convergence}: $\bF \leq \bM$;
\item \emph{$\partial$ is $\bF$-continuous}: $\bF(\partial T) \leq \bF(T)$ for all $T$.
\end{enumerate}
Those two axioms imply that, for any decomposition $T = R + \partial S$, one has:
\[\bF(T) \leq \bF(R) + \bF(\partial S) \leq \bF(R) + \bF(S) \leq \bM(R) + \bM(S).
\]
Thus, a candidate for the $\bF$-norm is
\[
\bF(T)
= \inf \left\{ \bM(R) + \bM(S) :
R \in \bP_m([0,1]^d),\;
S \in \bP_{m+1}([0,1]^d),\;
T = R + \partial S \right\}.
\]
This turns out to be the right choice, as (1) and (2) are easily established. The only difficult step is to prove $\bF(T) = 0 \implies T = 0$, which is carried out in~\cite[Chapter V, \S 12]{Whit}. The norm \(\bF\) is called the \emph{flat norm}. The following example shows that the flat norm encodes some sort of weak convergence, and why axiom~(1) is essential.

\begin{example}
  Consider the sequence of polyhedral $1$-chains $(T_n)$ defined by 
  \[
  T_n = \llbracket (0, 1/n), (1, 1/n) \rrbracket
  \]
  for all $n \geq 1$ (\ie{} $T_n$ is the oriented segment from $(0, 1/n)$ to $(1, 1/n)$). Calling $T = \llbracket (0, 0), (1, 0) \rrbracket$, we have
  \[
  T - T_n = R_n + \partial \llbracket S_n \rrbracket
  \] 
  where $S_n$ is the rectangle with vertices $(0, 0)$, $(1, 0)$, $(1, 1/n)$ and $(0, 1/n)$, oriented canonically, and
  \[
  R_n = \llbracket (1, 0), (1, 1/n) \rrbracket - \llbracket (0, 0), (0, 1/n) \rrbracket.
  \]
  Therefore, $\bF(T - T_n) \leq \bM(R_n) + \bM(S_n) = 2/n + 1/n \to 0$. Of course, $(T_n)$ does not tend to $T$ in mass norm.
\end{example}
\begin{defn}
\label{def:flat}
  The space of \emph{flat $m$-chains} is the completion of $\bP_m([0, 1]^d)$ under the flat norm $\bF$. It is denoted $\bF_m([0, 1]^d)$. We define the space of \emph{flat $m$-cochains} as the dual $\bF^m([0, 1]^d) = \bF_m([0, 1]^d)^*$.
\end{defn}
The boundary operator extends to a continuous linear map
\[
\partial \colon \bF_m([0, 1]^d) \to \bF_{m-1}([0, 1]^d).
\]
By taking the adjoint map, we define the \emph{coboundary} operator, or \emph{exterior derivative} operator
\[
d \colon \bF^{m-1}([0, 1]^d) \to \bF^m([0, 1]^d).
\]
It is also possible to extend the mass norm $\bM$ to $\bF_m([0, 1]^d)$ by relaxation
\[
\bM(T) = \inf_{(T_n)} \liminf_{n \to \infty} \bM(T_n) \in [0, \infty] \qquad \text{for } T \in \bF_m([0, 1]^d),
\]
where the infimum is taken over all sequences $(T_n)$ in $\bP_m([0, 1]^d)$ that converge to $T$ in flat norm,
and it can be proven that 
\[
\bF(T)
= \inf \left\{ \bM(R) + \bM(S) :
R \in \bF_m([0,1]^d),\;
S \in \bF_{m+1}([0,1]^d),\;
T = R + \partial S \right\}
\]
for all $T \in \bF_m([0, 1]^d)$.

To gain some intuition, we consider the two extreme cases $m = d$ and $m = 0$.

\begin{example}[$m = d$]
\label{ex:flatcodim0}
  In codimension $0$, namely for $d$-chains in $[0,1]^d$, flat chains admit a simple description.
A polyhedral \(d\)-chain 
\[
T = \sum_{i \in I} a_i \llbracket \sigma_i \rrbracket
\]
can be identified with the simple function
$f_T \colon [0,1]^d \to \R$, where
\[
f_T = \sum_{i \in I} \pm a_i \ind_{\sigma_i} \text{ a.e.}
\]
where we choose $+ a_i$ if $\llbracket \sigma_i \rrbracket$ is oriented canonically, and $-a_i$ otherwise.
With this identification, the mass of \(T\) satisfies
\[
\bM(T)
= \int_{[0,1]^d} |f_T(x)|\,dx
= \|f_T\|_{L^1}.
\]
Since there are no polyhedral chains of dimension \(d+1\),
the flat norm reduces to $\bF(T) = \bM(T)$
for all polyhedral \(d\)-chains \(T\).
Therefore, completing \(\bP_d([0,1]^d)\) with respect to
the flat norm amounts to completing the space of
functions that are a.e. piecewise constant on $d$-cells with respect to the \(L^1\) norm.
Consequently, every flat $d$-chain  can be represented by a function in
\(L^1([0,1]^d)\), and this representation is unique up to
equality almost everywhere. Thus we have isometric isomorphisms $\bF_d([0, 1]^d) \cong L^1([0, 1]^d)$ and $\bF^d([0, 1]^d) \cong L^\infty([0, 1]^d)$.
\end{example}

\begin{example}[$m = 0$]
  Each cochain $\omega \in \bF^0([0, 1]^d)$ determines a function $\Gamma(\omega) \colon [0, 1]^d \to \R$ defined by $\Gamma(\omega)(x) = \omega(\llbracket x \rrbracket)$ for all $x \in [0, 1]^d$. This function is Lipschitz continuous; indeed, for all $x, y \in [0, 1]^d$, we have:
  \begin{align*}
  |\Gamma(\omega)(y) - \Gamma(\omega)(x)| & = |\omega(\llbracket y \rrbracket - \llbracket x \rrbracket)| \\
  & = |\omega(\partial \llbracket x, y \rrbracket)| \\
  & \leq \|\omega\|_{\bF^0} \bF(\partial \llbracket x, y \rrbracket) \\
  & \leq \|\omega \|_{\bF^0} \bM(\llbracket x, y \rrbracket) \\
  & = \|\omega \|_{\bF^0} |y - x|.
  \end{align*}
  Conversely, if $f \in \rmLip([0, 1]^d)$, it determines a flat $0$-cochain $\Lambda(f)$. We first define $\Lambda(f)$ on the space of polyhedral $0$-chains $\bP_0([0, 1]^d)$ by setting
  \[
  \Lambda(f)(T) = \sum_{i \in I} a_i f(x_i) \quad \text{for } T = \sum_{i \in I} a_i \llbracket x_i \rrbracket \in \bP_0([0, 1]^d).
  \]
  It is clear that $|\Lambda(f)(T)| \leq \|f\|_\infty \sum_{i \in I} |a_i| = \|f\|_\infty \bM(T)$. Furthermore, for any polyhedral $1$-chain $S = \sum_{i \in I} a_i \llbracket x_i, y_i \rrbracket$, we have
  \[
  |\Lambda(f)(\partial S)| = \left| \sum_{i \in I} a_i (f(y_i) - f(x_i)) \right| \leq \rmLip f \sum_{i \in I} |a_i| |y_i - x_i| = \rmLip f \bM(S).
  \]
  Thus, for any decomposition $T = R + \partial S$, one has
  \[
  |\Lambda(f)(T)| \leq |\Lambda(f)(R)| + |\Lambda(f)(\partial S)| \leq \max\{ \|f\|_\infty, \rmLip f\} \left(\bM(R) + \bM(S) \right).
  \]
  By passing to the infimum over all such decompositions, one obtains $|\Lambda(f)(T)| \leq \max\{ \|f\|_\infty, \rmLip f\} \bF(T)$. Consequently, $\Lambda(f)$ extends by density to a flat $0$-cochain on the whole of $\bF_0([0,1]^d)$. This proves that the spaces $\bF^0([0, 1]^d)$ and $\rmLip([0, 1]^d)$ are isometrically isomorphic.
  
  This in turn provides a description of flat $0$-chains. We can identify a polyhedral $0$-chain $T = \sum_{i \in I} a_i \llbracket x_i \rrbracket$ with the atomic measure $\mu_T = \sum_{i \in I} a_i \delta_{x_i}$. The flat distance between two polyhedral $0$-chains $S$ and $T$ is then precisely the Kantorovich-Rubinstein distance:
  \[
  \bF(T - S) = \sup \left\{ \int_{[0, 1]^d} f \, d(\mu_T - \mu_S) : f \in \rmLip([0, 1]^d) \text{ and } \max \{\|f\|_\infty, \rmLip f \} \leq 1 \right\}.
  \]
  The space $\bF_0([0, 1]^d)$ is also the completion of the space of signed Borel measures $\scrM([0, 1]^d)$ under the flat metric. It is related to the Lipschitz-free space over $[0, 1]^d)$, a canonical predual of $\rmLip_0([0, 1]^d) = \{f \in \rmLip([0, 1]^d) : f(0) = 0\}$, an object that has attracted a lot of attention in the study of Banach spaces and metric geometry, see~\cite{depauw_quantified_2025}.
\end{example}

If \( m \in \{1,\dots,n-1\} \), flat \(m\)-chains can be quite general objects and are often difficult to describe explicitly. For instance, consider a closed Jordan curve in \([0,1]^2\) enclosing a domain \(\Omega\). This curve can be represented as the \(1\)-chain $\partial \llbracket \ind_{\Omega} \rrbracket$ (where $\llbracket \ind_\Omega \rrbracket$ denotes the flat $2$-chain associated to the $L^1$ indicator function $\ind_\Omega$).
However, a continuous oriented curve need not define a flat chain. In particular, flat chains cannot, in general, be pushed forward by arbitrary continuous maps; the push-forward is well defined only for Lipschitz maps.

There is a representation theorem due to Wolfe for flat cochains, which is consistent with the previous examples. A \emph{flat $m$-form} is a $L^\infty$ differential $m$-form $\omega \in L^\infty([0, 1]^d ; \wedgeop^m \R^d)$ that has an $L^\infty$ distributional exterior derivative $d\omega \in L^\infty([0, 1]^d  ; \wedgeop^{m+1} \R^d)$. The latter condition means that:
\[
\int_{[0, 1]^d} \omega \wedge d\eta = (-1)^{m+1} \int_{[0, 1]^d} d\omega \wedge \eta
\]
for every smooth $(d - m - 1)$-form $\eta$ on $\R^d$ that is compactly supported in $[0, 1]^d$. We equip the space of flat $m$-forms with the norm $\| \omega \|_{\text{flat}} = \max \{ \|\omega\|_\infty, \| d \omega \|_\infty \}$.

\begin{thm}[Wolfe]
\label{thm:Wolfe}
  There is an isometric isomorphism between the space of flat $m$-cochains and the space of flat $m$-forms.
\end{thm}

Besides Whitney's treatise, we refer to~\cite[4.1.19]{Fede} for a proof of Theorem~\ref{thm:Wolfe}. As a consequence, flat \(m\)-forms can be paired with flat \(m\)-chains by integration. We do not describe here the precise definition of this pairing. Let us simply note that it is not straightforward: for example, given a flat \(1\)-form \(\omega\) on \([0,1]^2\), it is not immediately meaningful to integrate \(\omega\) over an oriented segment \(\llbracket a,b\rrbracket\). Indeed, \(\omega\) is only defined almost everywhere with respect to Lebesgue measure, whereas the segment \([a,b]\) has Lebesgue measure zero.

Another consequence of Wolfe’s theorem is that it allows one to define the wedge product of flat cochains. Indeed, by the theorem, flat \(m\)-cochains can be identified with flat \(m\)-forms. Since flat \(m\)-forms admit an almost everywhere pointwise definition, their wedge product is well defined almost everywhere, and thus induces a well-defined wedge product at the level of flat cochains.

\subsection{Normal currents and charges in middle dimension}
\label{subsec:normal}

We define a \emph{normal $m$-current} to be a flat chain $T \in \bF_m([0, 1]^d)$ whose \emph{normal mass}
\[
\bN(T) = \bM(T) + \bM(\partial T)
\]
is finite. Normal $m$-currents should be thought of as the most general oriented \(m\)-dimensional domains with multiplicity, whose volume and the volume of their boundary are both finite. We denote by $\bN_m([0, 1]^d)$ the space of normal $m$-currents, normed by the normal mass. We note that $\partial$ restricts to a continuous linear map $\bN_m([0, 1]^d) \to \bN_{m-1}([0, 1]^d)$.

\begin{example}[Case $m = d$]
We already saw in Example~\ref{ex:flatcodim0} that the condition \(\bM(T) < \infty\) corresponds to the existence of a density \(u \in L^1([0,1]^d)\). One can further show that the condition \(\bM(\partial T) < \infty\) is equivalent to the requirement that the extension by zero of \(u\) belongs to \(BV([0,1]^d)\). More precisely, \(\bM(\partial T)\) coincides with the total variation of \(u\). Altogether, this yields the identification $\bN_d([0,1]^d) \cong BV([0,1]^d)$.
\end{example}

A fundamental property of the space $\bN_m([0,1]^d)$ is the
Federer-Fleming compactness theorem \cite[4.2.17]{Fede}, which asserts that its unit ball
is compact with respect to the flat norm $\bF$. This feature makes
$\bN_m([0,1]^d)$ well suited for variational problems in geometric
measure theory, such as the Plateau problem; however, we will not
pursue this aspect here. 

Compactness of the unit ball is also a strong indication that a Banach space is a dual space. Indeed, $\bN_m([0,1]^d)$ is a dual space, with a canonical
predual $\bCH^m([0,1]^d)$, the space of $m$-charges. The situation is entirely
analogous to the duality between $CH([0,1]^d)$ and $BV([0,1]^d)$
described in Theorem~\ref{thm:dualCH}. The following definition of a charge is due to De Pauw, Moonens, and Pfeffer, who initiated the study of charges in~\cite{DePaMoonPfef}.

\begin{defn}
\label{def:chm}
  An \emph{$m$-charge} over $[0, 1]^d$ is a linear functional $\omega \colon \bN_m([0, 1]^d) \to \R$ that is continuous in the following sense: for all sequences $(T_n)$ in $\bN_m([0, 1]^d)$ that converge to $0$ in flat norm with $\sup_n \bN(T_n) < \infty$, we have $\omega(T_n) \to 0$. The space of $m$-charges is denoted $\bCH^m([0, 1]^d)$. 
\end{defn}

The continuity requirement in this definition extends that of Proposition~\ref{prop:chargeCONT}. Indeed, in codimension zero, $d$-charges coincide with strong charges in the sense of Definition~\ref{def:CH01d} (thus, $m$-charges generalize strong charges, but the adjective ``strong'' is no longer used for them).
 This follows from the identification of $\bN_d([0,1]^d)$ with $BV([0,1]^d)$. One can show that $\bCH^m([0,1]^d)$ is a closed linear subspace of the dual space $\bN_m([0,1]^d)^*$. We therefore endow $\bCH^m([0,1]^d)$ with the operator norm
\[
\|\omega\|_{\bCH^m}= \sup\left\{ \omega(T) : T \in \bN_m([0,1]^d)\ \text{and}\ \bN(T)\leq 1 \right\}.
\]
Observe that flat $m$-cochains, when restricted to the dense subspace $\bN_m([0, 1]^d)$ of $\bF_m([0, 1]^d)$, define $m$-charges on $[0, 1]^d$. Indeed, $\bF$-continuity is stronger than the continuity required in Definition~\ref{def:chm}, so these cochains should be regarded as very regular objects. Recall that flat $0$-cochains correspond to Lipschitz functions. In the following example, we focus on $0$-charges.

\begin{example}[Case $m = 0$]
\label{ex:ch0}
   Each $0$-charge $\omega$ determines a function $\Gamma(\omega)$ defined by $\Gamma(\omega)(x) = \omega (\llbracket x \rrbracket)$ for all $x \in [0, 1]^d$. This function is continuous, indeed, let $(x_n)$ be a sequence that converges to $x \in [0, 1]^d$. The sequence of normal $0$-currents $(\llbracket x_n \rrbracket)$ converges to $\llbracket x \rrbracket$ in flat norm, with uniformly bounded normal masses, hence $\Gamma(\omega)(x_n) \to \Gamma(\omega)(x)$. This defines an operator
   \[
   \Gamma \colon \bCH^0([0, 1]^d) \to C([0, 1]^d).
   \]
   We admit that it is injective and continuous. Let us prove surjectivity.
   
Let $g \in C([0,1]^d)$ be a continuous function. Define
\[
\omega \colon \bP_0([0,1]^d) \to \R : 
\sum_{i \in I} a_i \llbracket x_i \rrbracket \mapsto \sum_{i \in I} a_i g(x_i),
\]
with the intention of extending it later.

For any $\varepsilon > 0$, choose a Lipschitz function $f$ on $[0,1]^d$ such that $\|g - f\|_\infty \le \varepsilon$. Setting $\theta = \max\{\|f\|_\infty, \rmLip f\}$,
we have, for every polyhedral $0$-chain $T = \sum_{i \in I} a_i \llbracket x_i \rrbracket$,
\[
|\omega(T)| \leq \left|\sum_{i \in I} a_i f(x_i)\right| + \left|\sum_{i \in I} a_i (g(x_i) - f(x_i))\right|.
\]
This gives
\begin{equation}
\label{eq:omegaT}
\forall \varepsilon > 0, \ \exists \theta \ge 0 \text{ such that } \forall T \in \bP_0([0,1]^d), \qquad 
|\omega(T)| \le \theta \bF(T) + \varepsilon \bN(T).
\end{equation}
Property~\eqref{eq:omegaT} allows us to extend $\omega$ to $\bN_0([0,1]^d)$ by
\[
\omega(T) = \lim_{n \to \infty} \omega(T_n).
\]
for any sequence $(T_n)$ of polyhedral $0$-chains that $\bF$-converges to $T$ with uniformly bounded normal masses. The sequence $(\omega(T_n))$ forms a Cauchy sequence, so the limit exists. The bound~\eqref{eq:omegaT} then holds for all $T \in \bN_0([0,1]^d)$, and from that we can deduce that $\omega$ satisfies the continuity property of $0$-charges. Clearly, $g = \Gamma(\omega)$.
\end{example}

More generally, continuous differential $m$-forms can be viewed as $m$-charges.
To make this precise, it is convenient to regard flat $m$-chains $T$ as currents,
following Federer's approach. Indeed, every smooth form
$\omega \in C^\infty([0,1]^d; \wedgeop^m \R^d)$ is a flat $m$-form, so the pairing
$\langle T, \omega \rangle$ is well defined by Wolfe's theorem. The finiteness of the mass of $T$
corresponds to the inequality
\[
|\langle T, \omega \rangle| \le \bM(T) \, \|\omega\|_\infty,
\]
which holds for all $\omega \in C^\infty([0,1]^d; \wedgeop^m \R^d)$.

By a vector-valued version of the Riesz representation theorem, there exists a
unique finite measure $\|T\|$ on $[0,1]^d$ and a Borel $m$-vector field
$\vec{T} \colon [0,1]^d \to \wedgeop_m \R^d$,
unique $\|T\|$-almost everywhere, such that
\begin{equation}
\label{eq:massfinie}
\langle T, \omega \rangle = \int_{[0,1]^d} \langle \vec{T}, \omega \rangle \, d\|T\|
\end{equation}
for all smooth forms $\omega \in C^\infty([0,1]^d; \wedgeop^m \R^d)$. At this point,
the right-hand side also makes sense for continuous $m$-forms. Therefore, for every
$\omega \in C([0,1]^d; \wedge^m \R^d)$, we define a linear map
$\Lambda(\omega) \colon \bN_m([0,1]^d) \to \R$ by
\[
\Lambda(\omega)(T) = \int_{[0,1]^d} \langle \vec{T}, \omega \rangle \, d\|T\|
\]
for all $T \in \bN_m([0,1]^d)$. By an argument analogous to that of Example~\ref{ex:ch0}, one can show that
$\Lambda(\omega)$ defines an $m$-charge. Moreover, the map
\[
\Lambda \colon C([0,1]^d; \wedgeop^m \R^d) \longrightarrow \bCH^m([0,1]^d)
\]
is a continuous injection. In the sequel, we will identify the continuous $m$-form $\omega$ with the $m$-charge $\Lambda(\omega)$.

\begin{defn}
The \emph{exterior derivative} of an $m$-charge $\omega$ is the $(m+1)$-charge $d\omega$ defined by
\[
d\omega(T) = \omega(\partial T)
\]
for all $T \in \bN_{m+1}([0, 1]^d)$. One checks easily that $d \circ d = 0$.
\end{defn}
Thus, the exterior derivative of continuous $m$-forms provides yet another class
of charges in the middle dimension. It turns out that, together with the previous
examples, this describes all $m$-charges, as the following theorem shows.
\begin{thm}[{\cite[Theorem~6.1]{DePaMoonPfef}}]
For every $\omega \in \bCH^m([0,1]^d)$, there exist continuous differential
forms $\omega_1$ and $\omega_2$ of degree $m$ and $m-1$, respectively, such that
\[
\omega = \omega_1 + d\omega_2.
\]
\end{thm}
This theorem basically asserts that $m$-charges form the smallest cochain complex generated by continuous differential forms. It plays a role analogous to Wolfe's theorem for flat cochains. However, unlike flat cochains, charges are too irregular for a meaningful notion of a wedge product to be defined between them. It is possible to prove a negative result that gives a precise formulation of this statement, but we will simply note the following intuition: if such a wedge product were possible, one could define integrals of the form $\int_0^1 f \, dg$ for arbitrary continuous functions $f, g$ on $[0,1]$ as the action of the $1$-charge $f \wedge dg$ on the normal current $\llbracket 0,1 \rrbracket$. This, however, is morally incorrect. The impossibility of defining $\wedge$ is closely related to the well-known difficulties in multiplying distributions.

\subsection{Fractional currents and  fractional charges} 
\label{subsec:ygt}
Taken together, fractional currents and fractional charges form a class of objects on which a Hölder-type non-smooth calculus can be developed. We indeed introduce an exterior calculus apparatus that allows one to treat Hölder differential forms, their exterior derivatives, and even their wedge products within the Young regime. We also explicitly characterize the corresponding chains, which represent the most general weakly defined surfaces over which such Hölder differential forms can also be integrated.

A closely related geometric integration theory was developed in \cite{ChandraSingh2025} by Chandra and Singh, where integration of non-smooth differential forms is also constructed. Their approach is complementary to ours: it is based on a simplicial sewing lemma, whereas our method is more Fourier-analytic in nature. Related results on the pullback of differential forms by Hölder maps can be found in \cite{HajlaszMirraSchikorra2025HoelderHeisenberg}. Another related line of work, allowing for a form of analysis on fractal sets, is provided by Harrison's theory of chainlet geometry; see~\cite{Harrison1998_continuity_integral_domain, Harrison2005_lectures_chainlet_geometry}.

Unlike the two previous settings, we begin by introducing the cochains of interest,
namely the $\gamma$-fractional $m$-charges. In any case, in geometric integration
theory the focus lies primarily on cochains rather than on chains, in contrast with
geometric measure theory.

The following definition is naturally inspired by that of $CH^\gamma([0, 1]^d)$ (Definition~\ref{def:chgamma}), the space of
top-dimensional fractional charges. The $L^1$ norm and the total variation of a
$BV$ function are replaced by their counterparts in arbitrary codimension, namely
the flat norm and the normal mass of a normal current.
\begin{defn}
  Let $\gamma \in \mathopen{(}0, 1\mathclose{)}$. A \emph{$\gamma$-fractional charge} is a linear map $\omega \colon \bN_m([0, 1]^d) \to \R$ for which there is a constant $C \geq 0$ such that
  \[
  |\omega(T)| \leq C \bN(T)^\gamma \bF(T)^{1 - \gamma}
  \]
  for all $T \in \bN_m([0, 1]^d)$. The best such constant $C \geq 0$ is the norm $\| \omega \|_{\bCH^{m,\gamma}}$. The space of $\gamma$-fractional charges is denoted $\bCH^{m,\gamma}([0, 1]^d)$.
\end{defn}

The space $\bCH^{m,\gamma}([0,1]^d)$ lies between $\bF^m([0,1]^d)$, the space of flat $m$-cochains, and $\bCH^m([0,1]^d)$. Flat $m$-cochains are the most regular objects; they can be viewed as $1$-fractional charges, if we allow $\gamma = 1$ in the preceding definition. In fact, one can show (unpublished result) that the space $\bCH^{m, \gamma}([0, 1]^d)$ arises via the real interpolation method:
\begin{equation}
\label{eq:chgammainter}
\bCH^{m, \gamma}([0, 1]^d) = \left[\bCH^m([0, 1]^d), \bF^m([0, 1]^d) \right]_{\gamma, \infty}.
\end{equation}

Using the inequalities $\bF(\partial T) \leq \bF(T)$ and $\bN(\partial T) \leq \bN(T)$, which hold for every normal current $T$, one readily verifies that the exterior derivative of a $\gamma$-fractional charge is again $\gamma$-fractional. In particular,
\[
0 \longrightarrow \bCH^{0, \gamma}([0, 1]^d)
\overset{d}{\longrightarrow} \bCH^{1, \gamma}([0, 1]^d)
\overset{d}{\longrightarrow} \cdots
\overset{d}{\longrightarrow} \bCH^{d, \gamma}([0, 1]^d)
\longrightarrow 0
\]
defines a cochain complex, an algebraic structure that was already present for flat cochains and charges.

The next examples show that those cochains provide a convenient setting for manipulating of Hölder differential forms.

\begin{example}[$\gamma$-fractional $0$-charges]
Recall the isomorphism $\Gamma \colon \bCH^0([0, 1]^d) \to C([0, 1]^d)$ from Example~\ref{ex:ch0}. Here we show that it restricts to an isomorphism $\bCH^{0, \gamma}([0, 1]^d) \to  C^\gamma([0, 1]^d)$. This is of course consistent with~\eqref{eq:chgammainter}, because
\[
C^\gamma([0, 1]^d) = \left[C([0, 1]^d), \rmLip([0, 1]^d) \right]_{\gamma, \infty}.
\]
Let $\omega \in \bCH^{0, \gamma}([0, 1]^d)$ and $x, y \in [0, 1]^d$. Then
\begin{align*}
|\Gamma(\omega)(y) - \Gamma(\omega)(x)|&  = |\omega(\llbracket y \rrbracket - \llbracket x \rrbracket)|\\ & \leq \|\omega\|_{\bCH^{0, \gamma}} \bN(\llbracket y \rrbracket - \llbracket x \rrbracket)^{1 - \gamma} \bF(\llbracket y \rrbracket - \llbracket x \rrbracket)^{\gamma} \\ & \leq 2^{1 - \gamma} \| \omega\|_{\bCH^{0, \gamma}} |y - x|^\gamma.
\end{align*}
and moreover one establishes more easily that $|\Gamma(\omega)(x)| \leq \|\omega\|_{\bCH^{0, \gamma}}$. All of this shows that $\Gamma$ restricts to a continuous linear operator
\[
\Gamma \colon \bCH^{0, \gamma}([0, 1]^d) \to C^\gamma([0, 1]^d)
\]
and we claim that it is a Banach space isomorphism. It remains to show that if $g \in C^\gamma([0, 1]^d)$, then the charge $\omega = \Lambda(g)$ is $\gamma$-fractional. Let $\varepsilon \leq 1$. It is possible to find a Lipschitz approximation $f$ of $g$, satisfying
\[
\|f\|_\infty \leq \|g\|_\infty, \|f -g\|_\infty \leq\varepsilon^\gamma [g]_{C^\gamma} \text{ and } \rmLip f \leq \varepsilon^{\gamma-1} [g]_{C^\gamma}.
\]
Such an $f$ can be obtained by convolution, or even by inf-convolution
\[
f \colon x \mapsto \inf \left\{ g(y) + \varepsilon^{\gamma - 1} [g]_{C^\gamma} |y - x| : y \in [0, 1]^d \right\}.
\]
For any polyhedral $0$-chain $T = \sum_{i \in I} a_i \llbracket x_i \rrbracket$, one has
\begin{align*}
|\omega(T)| & \leq \left|\sum_{i \in I} a_i f(x_i)\right| + \left|\sum_{i \in I} a_i (g(x_i) - f(x_i))\right| \\
& \leq \max \{ \|f\|_\infty , \rmLip f \} \bF(T) + \|g - f\|_\infty \bN(T) \\
& \leq \max \{ \|g\|_\infty, [g]_{C^\gamma} \} \left( \varepsilon^\gamma \bF(T) + \varepsilon^{\gamma - 1} \bN(T) \right)
\end{align*}
By choosing $\varepsilon = \bF(T) / \bN(T)$, one then obtains
\[
|\omega(T)| \leq 2  \max \{ \|g\|_\infty, [g]_{C^\gamma} \} \bN(T)^{1 - \gamma} \bF(T)^\gamma.
\]
This extends by density to all $T \in \bN_0([0, 1]^d)$ and we conclude that $\omega$ is $\gamma$-fractional.
\end{example}

\begin{example}[Hölder differential $m$-forms are fractional $m$-charges]
Using arguments similar to those in the previous examples, one can show that the map
\[
\Lambda \colon C\left([0,1]^d;\wedgeop^m\R^d\right)\to \bCH^m([0,1]^d)
\]
restricts to a continuous, injective linear operator
\[
\Lambda \colon C^\gamma\!\left([0,1]^d;\wedgeop^m\R^d\right)\to \bCH^{m,\gamma}([0,1]^d).
\]
In the sequel, we identify a Hölder differential $m$-form $\omega$ with its image $\Lambda(\omega)$.
\end{example}

\begin{quest}
Is every $\gamma$-fractional $m$-charge of the form $\omega_1 + d\omega_2$, where
$\omega_1$ and $\omega_2$ are $\gamma$-Hölder differential forms of degrees
$m$ and $m-1$, respectively?
\end{quest}

The most important result about fractional charges is the existence of a partially defined wedge product. Roughly speaking, if $\omega$ and $\eta$ are $\alpha$- and $\beta$-fractional charges of any degree, and their combined regularity satisfies the \emph{Young condition} $\alpha+\beta>1$, then their wedge product $\omega\wedge\eta$ can be defined. To describe its continuity properties, we first need to introduce the notion of weak* convergence in $\bCH^{m,\gamma}([0,1]^d)$. The following is a straightforward generalization of Definition~\ref{def:weak*}.
\begin{defn}
A sequence of $\gamma$-fractional charges $(\omega_n)$ is said to converge \emph{weakly*} to $\omega \in \bCH^{m, \gamma}([0, 1]^d)$ if it is bounded in $\bCH^{m, \gamma}([0, 1]^d)$ and $\omega_n \to \omega$ in $\bCH^m([0, 1]^d)$.
\end{defn}
In particular, for $m = 0$, if we identify (fractional) charges with (Hölder) continuous functions, we recover the standard concept of weak$^*$ convergence in $C^\gamma([0,1]^d)$.
\begin{thm}
\label{thm:wedge}
        Let $\alpha, \beta \in \mathopen{(} 0, 1 \mathclose{]}$ be such that $\alpha + \beta > 1$. There is a unique bilinear map
        \[
        \wedge \colon \bCH^{m,\alpha}([0, 1]^d) \times \bCH^{m', \beta}([0, 1]^d) \to
        \bCH^{m+m', \alpha + \beta - 1}([0, 1]^d)
        \]
        that
        \begin{enumerate}
            \item extends the wedge product between flat cochains (including smooth forms);
            \item is weak* continuous: if $\omega_n \to \omega$ weakly* in $\bCH^{m, \alpha}([0, 1]^d)$ and $\eta_n \to \eta$ weakly* in $\bCH^{m, \beta}([0, 1]^d)$ then 
            \[
            \omega_n \wedge \eta_n \to \omega \wedge \eta
            \]
             weakly* in $\bCH^{m+m', \alpha + \beta -1}([0, 1]^d)$.
        \end{enumerate}
\end{thm}

It can be seen that this wedge product encompasses several constructions. 
For instance, if $f$ and $\omega$ satisfy the regularity conditions of Theorem~\ref{thm:young}, the Young indefinite integral is
\[
(\mathrm{Y}) \int_\bullet f \, \omega = f \wedge \omega.
\]
(The top-dimensional result in Theorem~\ref{thm:young} provides better regularity for $f \wedge \omega$ than Theorem~\ref{thm:wedge}: both $\omega$ and $f \wedge \omega$ are $\gamma$-fractional; this improvement stems from the fact that $f$ has no ``differential'' part and therefore does not degrade the regularity of $f \wedge \omega$.) Züst's integral is also recovered: the wedge product
\[
f \wedge dg_1 \wedge \cdots \wedge dg_d
\]
of $\beta$, $\gamma_1, \dots, \gamma_d$-Hölder functions is well-defined under the condition \[\beta + \gamma_1 + \cdots + \gamma_d > d. \]
Moreover, it is $(\beta + \gamma_1 + \cdots + \gamma_d - d)$-fractional.

We now outline an idea of the proof of Theorem~\ref{thm:wedge}. Our approach is more ``Fourier-analytic'' in nature. 
Since the space $\bCH^{m,\gamma}([0,1]^d)$ is an interpolation space, abstract results from 
interpolation theory—specifically, the equivalence between the $K$- and $J$-methods of real 
interpolation, see~\cite[Section~3.3]{BerghLofstrom1976}—guarantee the existence of Littlewood-Paley-type decompositions for fractional 
charges. These decompositions will serve as our main tool. To avoid relying on abstract machinery, 
we instead present the following easy-to-prove result.
\begin{lem}
\label{lem:alp}
        Suppose $0 < \gamma < 1$. Let $(\omega_n)$ be a sequence in
  of flat $m$-cochains such that
  \[
  \|\omega_n\|_{\bF^m} \leq \kappa 2^{n(1 - \gamma)} \text{ and }
  \|\omega_n\|_{\bCH^m} \leq \frac{\kappa}{2^{n\gamma}}
  \]
  for some $\kappa \geq 0$ and for all $n$. Then $\sum_{n=0}^\infty
  \omega_n$ converges weakly* to a $\gamma$-fractional charge and
  \[
  \left\| \sum_{n=0}^\infty \omega_n \right\|_{\bCH^{m,\gamma}} \leq
  C \kappa 
  \]
  where $C = C(\gamma)$ is a constant.
  \end{lem}
  The fact that the series $\sum_{n=0}^\infty \omega_n$ does not converge strongly is reminiscent of the weak orthogonality properties of Littlewood–Paley components in harmonic analysis.
  
\begin{proof}
   Let $T \in \bN_m([0, 1]^d)$. In the following, we will just prove the convergence of the series $\sum_{n=0}^\infty \omega_n(T)$. The rest is left to the reader. We can estimate
  $|\omega_n(T)|$ in two different ways
  \[
  |\omega_n(T)| \leq \kappa 2^{n(1 - \gamma)} \bF(T) \text{ and } |\omega_n(T)| \leq
  2^{-n\gamma} \kappa \bN(T).
  \]
  Let $N$ be a nonnegative integer, to be determined later. We have
  \[
  \left| \sum_{k=0}^N \omega_n(T) \right| \leq \kappa \bF(T) \sum_{k=0}^N 2^{n(1 - \gamma)}  \leq C \kappa \bF(T) 2^{n(1 - \gamma)}
  \]
  and
  \[
  \left| \sum_{k=N+1}^\infty \omega_n(T) \right| \leq \kappa \bN(T)
  \sum_{k=N+1}^\infty 2^{-n\gamma} \leq C \kappa \bN(T) 2^{-N\gamma}.
  \]
  Since $\bF(T) \leq \bN(T)$, we choose $N$ to be a nonnegative
  integer such that
  \[
  2^{-(N+1)} \leq \frac{\bF(T)}{\bN(T)} \leq 2^{-N}
  \]
  Using the preceding inequalities, we obtain
  \[
  \left| \sum_{k=0}^\infty \omega_n(T) \right| 
   \leq C \kappa \bN(T)^{1 - \gamma} \bF(T)^\gamma. \qedhere
  \]
\end{proof}
  
\begin{proof}[Sketch of proof of Theorem~\ref{thm:wedge}]
Let $\omega \in \bCH^{m, \alpha}([0, 1]^d)$ and $\eta \in \bCH^{m', \beta}([0, 1]^d)$. One can construct ``Littlewood-Paley decompositions'' of $\omega$ and $\eta$,
\[
\omega = \sum_{n=0}^\infty \omega_n \quad \text{and} \quad \eta = \sum_{n=0}^\infty \eta_n,
\]
that satisfy the assumptions of Lemma~\ref{lem:alp}, with $\gamma = \alpha$ and $\gamma = \beta$ respectively. These decompositions can, for instance, be obtained by convolving $\omega$ and $\eta$ with appropriate kernels. Consider the formal series
\[
\omega_0 \wedge \eta_0 + \sum_{n=0}^\infty \left( \omega_{n+1} \wedge (\eta_{n+1} - \eta_n) + (\omega_{n+1} - \omega_n) \wedge \eta_n \right).
\]
Natural estimates for the $\bCH^m$ and $\bF^m$ norms of the wedge product of flat cochains show that this series satisfies the assumptions of Lemma~\ref{lem:alp} for $\gamma = \alpha + \beta - 1$. Hence, it defines an object that we denote by $\omega \wedge \eta$. This approach is reminiscent of paraproducts; however, there is no difficulty in handling a ``resonant term'' because we are in the Young regime.
\end{proof}

By definition, fractional charges act on normal currents. However, their action naturally extends to geometric objects that are more irregular than normal currents and that live within the space of flat chains. We refer to these objects as fractional currents, which form the second side of the picture. We now introduce their precise definition.

\begin{defn}
\label{def:fractCurrent}
A \emph{$\gamma$-fractional $m$-current} over $[0,1]^d$ is a flat $m$-chain $T \in \bF_m([0,1]^d)$ for which there exists a sequence of normal currents $(T_k)_{k \ge 0}$ such that
\begin{enumerate}
\item $T = \sum_{k=0}^\infty T_k$, with convergence in $\bF_m([0,1]^d)$;
\item $\sum_{k=0}^\infty \bN(T_k)^{1-\gamma}\,\bF(T_k)^\gamma < \infty$.
\end{enumerate}
The space of $\gamma$-fractional $m$-currents is denoted by $\bF_m^\gamma([0,1]^d)$ and is equipped with the norm
\[
\bF^\gamma(T)
  = \inf \sum_{k=0}^\infty \bN(T_k)^{1-\gamma}\,\bF(T_k)^\gamma,
\]
where the infimum is taken over all decompositions of $T$ as above.
\end{defn}
Here, the assumptions become less restrictive as $\gamma$ increases.
In the limiting case $\gamma = 1$, being $1$-fractional simply amounts to being a flat chain.
In fact, one can show that the space $\bF^\gamma([0,1]^d)$ arises as an interpolation space
between $\bN_m([0,1]^d)$ and $\bF([0,1]^d)$, obtained once again via the real interpolation method:
\[
\bF_m^\gamma([0, 1]^d) = \left[ \bN_m([0, 1]^d), \bF_m([0, 1]^d) \right]_{\gamma, 1}.
\]
Comparing with~\eqref{eq:chgammainter}, we guess that $\bCH^{m, \gamma}([0, 1]^d)$ is the dual of $\bF_m^\gamma([0, 1]^d)$, \cite[Theorem~6.4]{Boua2}. In fact, the pairing between $\omega \in \bCH^{m, \gamma}([0, 1]^d)$ and $T \in \bF_m^\alpha([0, 1]^d)$ is
\[
\langle T, \omega \rangle = \sum_{k=0}^\infty \omega(T_k)
\]
where $(T_k)_{k \geq 0}$ is a decomposition of $T$ as in the definition of $\gamma$-fractional current.

We believe that the most insightful way to understand fractional currents is by examining the extreme cases $m = 0$ and $m = d$.
\begin{example}[Case $m = 0$]
The space $\bF_0^\gamma([0,1]^d)$ can be seen as the canonical predual of $C^\gamma([0,1]^d)$. 
It is obtained by completing $\bP_0([0,1]^d)$ with respect to the norm
\[
\sum_{i \in I} a_i \llbracket x_i \rrbracket 
 \mapsto
\sup \left\{ \sum_{i \in I} a_i g(x_i) : g \in C^\gamma([0,1]^d), \; \max \{\|g\|_\infty, [g]_{C^\gamma}\} \le 1 \right\},
\]
which is equivalent to the $\bF^\gamma$-norm.  

Analogous to the case of flat $0$-chains, the space $\bF_0^\gamma([0,1]^d)$ is almost a Lipschitz-free space, but this time over the snowflaked metric space $([0,1]^d, d_{\mathrm{Eucl}}^\gamma)$.
\end{example}

\begin{example}[Case $m = d$]
\label{ex:fractSobolev}
  The $\gamma$-fractional $d$-currents are those that have a density in the interpolation space
  \[
  \left[ BV([0, 1]^d), L^1([0, 1]^d) \right]_{\gamma, 1} = \left[ L^1([0, 1]^d), BV([0, 1]^d) \right]_{1 - \gamma, 1} = W^{1, 1 - \gamma}([0, 1]^d)
  \]
  of \emph{fractional Sobolev functions}. This is the space of all integrable functions $u \colon \R^d \to \R$ such that $u = 0$ a.e. outside $[0, 1]^d$ and the \emph{Gagliardo norm}
  \[
  \|u \|_{W^{1, 1 - \gamma}} = \int_{\R^d} \int_{\R^d} \frac{|u(x) - u(y)|}{|x - y|^{d+1- \gamma}} \,dx \,dy 
  \]
  is finite. Traditionally, fractional Sobolev spaces are introduced as interpolation spaces between $L^1$ and $W^{1,1}$. 
However, $W^{1,1}$ can be replaced by its Gagliardo completion $BV$ (see \cite[Chapter~5, Section~1]{BennettSharpley1988}). 
This observation was made in~\cite{PoncSpec}, where several interesting consequences are derived.

One of the key results for fractional currents is the existence of a pushforward operation under a Hölder map $f$ (and dually, a pullback construction is available for fractional charges). This pushforward, however, is subject to certain Young-type restrictions on the Hölder exponent of $f$.

More precisely, let $f \colon [0, 1]^d \to [0, 1]^{d'}$ be $\alpha$-Hölder continuous. Then $f$ defines a pushforward operator (see \cite[Theorem~8.1]{Boua2})
\[
f_\# \colon \bF_m^{\gamma}([0, 1]^d) \to \bF_m^{\gamma'}([0, 1]^{d'})
\]
provided that the Hölder exponents satisfy
\begin{equation}
\label{eq:pushexp}
\frac{m+\gamma}{\alpha} = m + \gamma' < m + 1.
\end{equation}
Moreover, the pushforward operator $f_\#$ commutes with the boundary operator $\partial$. For instance, in the case $m = 1$ and $\gamma = 0$ (the limiting case $\gamma = 0$ corresponds to normal $1$-currents), this condition requires that $\alpha > 1/2$. 

Consequently, any oriented $\alpha$-Hölder continuous curve parameterized over $[0, 1]$ can be interpreted as a fractional current $f_\# \llbracket 0, 1 \rrbracket$ within the admissible range of exponents. Part of the literature on stochastic currents is devoted to understanding $f_\# \llbracket 0, 1 \rrbracket$ when $f$ falls below the required regularity ($\alpha \leq 1/2$)—for instance, when $f$ is a sample path of Brownian motion. In such irregular settings, $f_\# \llbracket 0, 1 \rrbracket$ is typically (almost surely) a de Rham current that is not a flat $1$-chain. We refer to~\cite{FlandoliGubinelliGiaquintaTortorelli2005, FlandoliGubinelliRusso2009}.

Equation~\eqref{eq:pushexp} provides an interpretation for the fractionality exponent $\gamma$. A fractional $m$-current, while being $m$-dimensional in an algebraic sense (in that it can be integrated against $m$-forms), is characterized by the property that its ``fractal dimension'' (a loosely undefined notion) lies between $m$ and $m+1$. The upper bound $m+1$ arises because flat $m$-chains are at worst boundaries of $(m+1)$-chains. For a $\gamma$-fractional current, its fractal dimension is at most $m+\gamma$, while the fractal dimension of its boundary is at most $m-1+\gamma$. Thus, the parameter $\gamma$ can be understood as an upper bound on the fractional part of the fractal dimension, in accordance with~\eqref{eq:pushexp}.

\end{example}